\newif\ifams\amstrue

\newif\ifdebug\debugtrue
\debugfalse

\newcommand\thetitle{Discreteness of spectrum for the \texorpdfstring{$\dbar$}{d-bar}-Neumann Laplacian on manifolds of bounded geometry}
\newcommand\theshorttitle{The \texorpdfstring{$\dbar$}{d-bar}-Neumann problem on manifolds of bounded geometry}
\newcommand\theauthor{Franz Berger}
\newcommand\theemail{franz.berger2@univie.ac.at}
\newcommand\theaddress{Fakult\"at f\"ur Mathematik, Universit\"at Wien, Oskar-Morgenstern-Platz 1, A-1090 Wien, Austria}
\newcommand\thethanks{This work was supported by the Austrian Science Fund (FWF): P28154.}

\pdfoutput=1

\newcommand{\sectionnumdepth}{1}
\newcommand{\subsectionnumdepth}{2}

\ifams
 \documentclass[reqno,11pt]{amsart}
 \usepackage[BCOR=8mm,DIV=13,headinclude=true,footinclude=false]{typearea}
   \else
 \documentclass[3p,11pt,sort&compress]{elsarticle}
\fi

\usepackage{versions}

\usepackage[T1]{fontenc}
\usepackage{type1ec}
\usepackage[utf8]{inputenc}
\usepackage[english]{babel}
\usepackage{csquotes}
\usepackage[all]{foreign}  \usepackage[shortcuts]{extdash} 

\usepackage[usenames,dvipsnames]{xcolor} 

\ifams
 \usepackage[
  unicode=true,
  pdfusetitle,
  bookmarks=true,
  bookmarksnumbered=false,
  bookmarksopen=false,
  breaklinks=false,
  pdfborder={0 0 1},
  backref=false,
  hypertexnames=false,
 ]{hyperref}
\else
 \usepackage[
  unicode=true,
  colorlinks=true,
 ]{hyperref}
\fi

\usepackage{amsmath}
\usepackage{amssymb}
\usepackage{amsthm}
\usepackage{thmtools}
\usepackage{mathtools}
\usepackage{esint}
\usepackage{tikz-cd}

\usepackage[
 activate={true,nocompatibility},
 final,
 tracking=true,
 kerning=true,
 spacing=true,
 factor=1100,
 stretch=10,
 shrink=10
]{microtype}
\SetTracking{encoding={*}, shape=sc}{0} \microtypecontext{spacing=nonfrench}

\usepackage{lmodern}
\usepackage{mathrsfs} \usepackage[mathcal]{euscript} 

\ifams
 \usepackage[
  backend=biber,
  hyperref=true,
  url=false,
  isbn=false,
  eprint=true,
  doi=true,
  backref=false,
  style=alphabetic,
  natbib=true,
  giveninits=true,
  sortcites=true
  ]
  {biblatex}
 \renewcommand*{\bibfont}{\footnotesize} \else
 \usepackage{doi} \fi

\numberwithin{equation}{section}
\numberwithin{figure}{section}

\declaretheorem[name=Theorem,style=plain,numberwithin=section]{thm}
\declaretheorem[name=Theorem,style=plain,numbered=no]{thm*}
\declaretheorem[name=Proposition,style=plain,sibling=thm]{prop}
\declaretheorem[name=Lemma,style=plain,sibling=thm]{lem}
\declaretheorem[name=Corollary,style=plain,sibling=thm]{cor}
\declaretheorem[name=Example,style=definition,sibling=thm,
qed={$\lozenge$}
]{example}
\declaretheorem[name=Definition,style=definition,sibling=thm]{defn}
\declaretheorem[name=Remark,style=remark,sibling=thm,
qed={$\lozenge$}
]{rem}
\declaretheorem[name=Remark,style=remark,numbered=no]{rem*}

\declaretheorem[name=Theorem,style=plain,numbered=yes]{thmx}

\declaretheorem[name=Corollary,style=plain,sibling=thmx]{corx}

\usepackage{cleveref}
\crefname{thm}{Theorem}{Theorems}
\crefname{thmx}{Theorem}{Theorems}
\crefname{prop}{Proposition}{Propositions}
\crefname{lem}{Lemma}{Lemmas}
\crefname{cor}{Corollary}{Corollaries}
\crefname{corx}{Corollary}{Corollaries}
\crefname{example}{Example}{Examples}
\crefname{defn}{Definition}{Definitions}
\crefname{rem}{Remark}{Remarks}
\crefname{claim}{Claim}{Claims}
\crefname{assum}{Assumption}{Assumptions}
\crefname{enumi}{}{}
\crefname{enumii}{}{}
\crefname{enumiii}{}{}
\crefname{equation}{}{}

\usepackage[inline]{enumitem}
\newcommand{\enumlabelformat}{\roman}
\newcommand{\enumlabelfont}[1]{#1}
\newlength{\thelabelsep}
\setlist{labelsep=\thelabelsep}
\setlist[enumerate]{font=\enumlabelfont,label=(\enumlabelformat*),leftmargin=2.5em}
\setlist[itemize]{leftmargin=2.5em,label=$-$}

\newcounter{inlineenum}
\renewcommand{\theinlineenum}{\enumlabelformat{inlineenum}}
\newenvironment{inlineenum}
  {\setcounter{inlineenum}{0}   \renewcommand{\item}{\refstepcounter{inlineenum}{(\theinlineenum)\hspace{\thelabelsep}}}
  }
  {\ignorespacesafterend}

\newcommand\numberthis{\addtocounter{equation}{1}\tag{\theequation}}

\newcommand{\overbar}[1]{\mkern 1mu\overline{\mkern-1mu#1\mkern-1mu}\mkern 1mu} \newcommand{\ol}[1]{\overline{#1}}

\newcommand{\R}{\mathbb{R}}
\newcommand{\Cplx}{\mathbb{C}}
\newcommand{\Cplxi}{i}
\newcommand{\N}{\mathbb{N}}

\newcommand{\bsecsymbol}{\Gamma}
\newcommand{\bsec}{\bsecsymbol}
\newcommand{\bsecc}{\bsecsymbol_c}
\newcommand{\cc}{{cc}}
\newcommand{\bseccc}{\bsecsymbol_{\cc}}

\newcommand{\interior}[1]{{#1}^\circ} 

\newcommand{\fadj}{t}

\newcommand{\arghere}{{\_}}

 \newcommand{\id}[1]{\operatorname{id}_{#1}} \newcommand{\img}[1]{\operatorname{img}(#1)} \newcommand{\domsymb}{\operatorname{dom}}
\newcommand{\dom}[1]{\domsymb(#1)} \newcommand{\graph}[1]{\operatorname{Graph}(#1)}  \DeclareMathOperator{\dimsymb}{dim}
\renewcommand{\dim}[2][]{\dimsymb_{#1}(#2)} \newcommand{\rank}[1]{\operatorname{rank}({#1})} \newcommand{\dist}[3][]{\operatorname{dist}_{#1}(#2,#3)} 
\newcommand{\trace}{\operatorname{tr}}

\newcommand{\re}{\operatorname{Re}}   

  \newcommand{\spec}[2][]{{\sigma_{#1}(#2)}} \newcommand{\essspec}[1]{{\sigma_e(#1)}}    

 \newcommand{\bdopsymb}{\mathscr L}
\newcommand{\bdop}[1]{{\bdopsymb(#1)}}

\newcommand{\dbar}{{\smash{\overbar{\partial}}}}  

\usepackage{xparse}
\NewDocumentCommand{\dbarN}     { O{} O{} }{N_{#1}^{#2}}       \NewDocumentCommand{\clapl} { O{} O{} }{\square_{#1}^{#2}} \NewDocumentCommand{\minsol}{ O{} O{} }{S_{#1}^{#2}}

\newcommand{\hstar}[1][]{\smash{\ol\star}^{#1}} \newcommand{\opwedge}[1]{\mathbin{\wedge_{#1}}}
\newcommand{\evwedge}{\opwedge{\mathrm{ev}}}

 \newcommand{\supp}{\operatorname{supp}}     
  \newcommand{\en}[1]{\operatorname{End}(#1)} \newcommand{\ins}[1]{\operatorname{\mathnormal{i}}_{#1}} \newcommand{\extpsymb}{\operatorname{\varepsilon}}
\newcommand{\extp}[1]{\extpsymb(#1)}  \newcommand{\symb}[2][]{\operatorname{Symb}_{#1}({#2})}

\newcommand{\nakanoc}[2][]{\operatorname{Nak}_{#1}(#2)}

\newcommand{\minn}{s} \newcommand{\maxx}{w} 

\newcommand{\loc}{\mathrm{loc}}

\newcommand{\rinj}[1]{{r_{\mathrm{inj}}(#1)}}

\DeclareMathOperator*{\psum}{ \mathchoice
  {\sideset{}{^{\,\prime}}{\sum}}
  {\sum^{\prime}\!\!}
  {\sum^{\prime}\!\!}
  {\sum^{\prime}\!\!}
}

\newlength{\hilbcomparrow}
\usepackage{mathtools}

\makeatletter
\DeclareFontFamily{OMX}{MnSymbolE}{}
\DeclareSymbolFont{MnLargeSymbols}{OMX}{MnSymbolE}{m}{n}
\SetSymbolFont{MnLargeSymbols}{bold}{OMX}{MnSymbolE}{b}{n}
\DeclareFontShape{OMX}{MnSymbolE}{m}{n}{
    <-6>  MnSymbolE5
   <6-7>  MnSymbolE6
   <7-8>  MnSymbolE7
   <8-9>  MnSymbolE8
   <9-10> MnSymbolE9
  <10-12> MnSymbolE10
  <12->   MnSymbolE12
}{}
\DeclareFontShape{OMX}{MnSymbolE}{b}{n}{
    <-6>  MnSymbolE-Bold5
   <6-7>  MnSymbolE-Bold6
   <7-8>  MnSymbolE-Bold7
   <8-9>  MnSymbolE-Bold8
   <9-10> MnSymbolE-Bold9
  <10-12> MnSymbolE-Bold10
  <12->   MnSymbolE-Bold12
}{}
\let\llangle\@undefined
\let\rrangle\@undefined
\DeclareMathDelimiter{\llangle}{\mathopen}                     {MnLargeSymbols}{'164}{MnLargeSymbols}{'164}
\DeclareMathDelimiter{\rrangle}{\mathclose}                     {MnLargeSymbols}{'171}{MnLargeSymbols}{'171}
\makeatother
  
\ifdebug
\usepackage[inline]{showlabels}

\fi

\ifams
\fi

\let\defemph\emph

\ifdebug
\fi

\newcommand{\volform}[1]{{\operatorname{vol}_{#1}}}
\newcommand{\volg}{\volform{g}}
\renewcommand{\ins}[1]{\operatorname{ins}_{#1}}
\renewcommand{\fadj}{\dagger}

\allowdisplaybreaks[1]

\ifams
\fi

\begin{document}

\ifams\else
\bibliographystyle{abbrvnat}
\fi

\ifams\else
\begin{frontmatter}
 \fi
 
 \ifams
 \title[\theshorttitle]{\thetitle}
 \author{\theauthor}
 \address{\theaddress}
 \email{\href{mailto:\theemail}{\theemail}}
 \thanks{\thethanks}
 \else
 \title\thetitle  \author{Franz Berger}  \ead{\theemail}
 \address{\theaddress}
 \journal{\dots}
 \fi

 \begin{abstract}
  For a Hermitian holomorphic vector bundle over a Hermitian manifold, we consider the Dolbeault Laplacian with $\dbar$-Neumann boundary conditions, which is a self-adjoint operator on the space of square-integrable differential forms with values in the given holomorphic bundle.
  We argue that some known results on the spectral properties of this operator on pseudoconvex domains in $\mathbb C^n$ continue to hold on K\"ahler manifolds satisfying certain bounded geometry assumptions.
  In particular, we will consider the Dolbeault complex for forms with values in a line bundle, where known results from magnetic Schr\"odinger operator theory can be applied.
 \end{abstract}

 \ifams\else
\end{frontmatter}
\fi

\ifams
\maketitle
\fi
\tableofcontents

\section{Introduction and overview of results}

Let $X$ be a complex manifold, and suppose that $M \subseteq X$ is the closure of a smoothly bounded open subset $\interior M$ of $X$, with (possibly empty) boundary $\partial M$.
Let $E \to M$ be a holomorphic vector bundle, meaning that $E$ is defined in some open neighborhood of $M$ and holomorphic on this neighborhood.
For $0 \leq p \leq n$, the \defemph{Dolbeault complex}
\begin{equation}
 \label{eq:dolbeault_complex}
 0 \to \Omega^{p,0}(M,E) \xrightarrow{\dbar^E} \Omega^{p,1}(M,E) \xrightarrow{\dbar^E} \dotsb \xrightarrow{\dbar^E} \Omega^{p,n}(M,E) \to 0,
\end{equation}
with $\Omega^{p,q}(M,E)$ the space of smooth $E$-valued $(p,q)$-forms, generalizes the Wirtinger derivative $\tfrac{d}{d\ol z}$ from single variable complex analysis.
Choosing Hermitian metrics on $X$ and on $E$ gives, in the spirit of Hodge theory, rise to the corresponding Dolbeault Laplacian $\dbar^E\dbar^{E,\fadj} + \dbar^{E,\fadj}\dbar^E$, with $\dbar^{E,\fadj}$ denoting the formal adjoint.
The \defemph{Dolbeault Laplacian with $\dbar$-Neumann boundary conditions} is the self-adjoint operator
\begin{equation*}
 \square^E \coloneqq \dbar^E_\maxx \dbar^{E,*}_\maxx + \dbar^{E,*}_\maxx \dbar^E_\maxx
\end{equation*}
on $L^2_{\bullet,\bullet}(M,E)$, where $\dbar^E_\maxx$ is the weak extension of $\dbar^E$, see \cref{sec:strong_weak_extensions}, and  $\dbar^{E,*}_\maxx$ is the Hilbert space adjoint of $\dbar^E_\maxx$.
The associated boundary value problem is the so-called \defemph{$\dbar^E$-Neumann problem}.
Some more information on $\square^E$ is provided in \cref{sec:more_on_dbar}.
If we just consider $\Cplx$-valued forms (\ie $E$ is the trivial line bundle), then we omit the superscript $E$ and simply write $\square$.
We shall always assume that $M$ is complete for the chosen metric, because then the operators of interest will have cores consisting of smooth sections with compact support, see \cref{sec:essential_self-adjointness}.

The $\dbar$-Neumann problem is an important tool in the theory of several complex variables.
Its solution is used in arguments requiring the construction of holomorphic functions (or, more generally, sections of $E$) with prescribed properties.
In addition, there are spectral geometry type results for $\square^E$, at least in the case of some domains in $\Cplx^n$, which deduce geometric properties of the (boundary of) the domain in terms of the spectrum of the Laplacian, see \cite{Fu2008}.
For extensive surveys of the $\dbar$-Neumann problem, with a focus on bounded pseudoconvex domains in $\mathbb C^n$, see \cite{Straube2010,Chen2001}.

The goal of this article is to establish generalizations of a few facts concerning the discreteness of spectrum for $\square^E$ that were previously known in the setting of pseudoconvex domains in $\Cplx^n$ or in the ``weighted'' $\dbar$-Neumann problem on $\Cplx^n$, with plurisubharmonic weight function.

\subsection*{``Percolation'' of bounds on the essential spectrum}

One of the results of this paper is that, under certain pseudoconvexity assumptions on $\partial M$ and positivity of curvature requirements, the discreteness of spectrum of $\square^E$ ``percolates'' up the Dolbeault complex, in the sense that if $\square^E_{p,q}$ has discrete spectrum, then the same holds true for $\square^E_{p,q+1}$.
This property is well-known in the case of a bounded pseudoconvex domain $M$ in $\Cplx^n$, see \cite[Proposition~2.2]{Fu2008} or \cite[Proposition~4.5]{Straube2010}.
Moreover, this holds also for the ``weighted'' $\dbar$-problem on $\Cplx^n$, and where the weight is plurisubharmonic, see \cite{Haslinger2014}.
Here, by the ``weighted'' problem, we mean choosing $E$ to be the trivial line bundle on $\Cplx^n$, but with nontrivial Hermitian metric, so that there is $\varphi \colon \Cplx^n \to \R$ such that the $L^2$ norm becomes $\|f\|^2 = \int_{\Cplx^n} |f(z)|^2\,e^{-\varphi(z)}d\lambda(z)$ after identifying sections of $E$ with functions.
Here, $\lambda$ is Lebesgue measure.
For a general vector bundle, the condition of $\varphi$ being plurisubharmonic will have to be replaced by a curvature condition.

For domains in $\Cplx^n$, the proofs of the above rely on the fact that, if $\{w_j\}_{j=1}^n$ is a constant orthonormal frame field for $T^{0,1}M$, then the isometry $L^2_{p,q}(M,E) \to L^2_{p,q-1}(M,E)^{\oplus n}$ given by $u \mapsto \frac{1}{\sqrt q}(\ins{w_j}(u))_{j=1}^n$ satisfies $\sum_{j=1}^n Q^E_{p,q-1}(\ins{w_j}(u),\ins{w_j}(u)) \leq C Q^E_{p,q}(u,u)$, assuming the previously mentioned pseudoconvexity and curvature assumptions hold.
Here, $Q^E$ is the quadratic form associated to $\square^E$, and $\ins{w_j}$ is the insertion operator on (differential) forms.
If $M$ is a Hermitian manifold, we do not have global frames for $T^{0,1}M$ available, so we have to use local frames and patch the results together.
Moreover, the derivatives of the frame elements will have to be controlled.
This patching procedure works if $X$ is of \defemph{$1$-bounded geometry} (to be discussed in \cref{sec:bounded_geometry}), and we have the following result:

\begin{thmx}
 \label{spectral_percolation_specific}
 Let $X$ be K\"ahler and of $1$-bounded geometry, suppose $M$ is $q$-Levi pseudoconvex, and assume that $E \to M$ is $q$-Nakano lower semibounded.
 Then $\inf\essspec{\square^E_{p,q-1}} \leq 2\inf\essspec{\square^E_{p,q}} + C$, with $C \geq 0$ depending on $p$, $q$, $n$, and on the geometries of $M$ and $E$.
 In particular, if $\square^E_{p,q-1}$ has discrete spectrum, then so does $\square^E_{p,q}$.
\end{thmx}

The notion of $q$-Levi pseudoconvexity will be discussed in \cref{sec:complex_geometry}, and a Hermitian holomorphic vector bundle $E \to M$ is called \defemph{$q$-Nakano lower semibounded} (with $q \geq 1$) if there is $c \in \R$ such that
\begin{equation}
 \label{eq:nakano_lower_semibounded}
 \sum_{j,k=1}^n \langle R^E(w_j,\ol w_k) \ins{\ol w_j}(u), \ins{\ol w_k}(u) \rangle \geq c |u|^2
\end{equation}
holds for all $u \in \Lambda^{0,q}T^*_xM \otimes E_x$ and all $x \in M$, and where $R^E$ is the curvature of the Chern connection on $E$.
Note that if \cref{eq:nakano_lower_semibounded} holds on $\Lambda^{0,q}T^*M \otimes E$, then it is also true on $\Lambda^{p,q}T^*M \otimes E$ for $0 \leq p \leq n$.
The largest possible constant $c$ in \cref{eq:nakano_lower_semibounded} is denoted by $\nakanoc[q]{E}$.
It is easy to see that $E$ is $1$-Nakano lower semibounded with $\nakanoc[1]{E} \geq 0$ if and only it is \defemph{Nakano semipositive} in the sense of \cite{Nakano1955}: for every $x \in M$, we have
\begin{equation}
 \label{eq:nakano_semi_positive}
 \sum_{j,k,\alpha,\beta} \big\langle R^E\big(\tfrac{\partial}{\partial z_j},\tfrac{\partial}{\partial \ol z_k}\big) e_\alpha,e_\beta \big\rangle \, u_{j,\alpha} \ol{u_{k,\beta}} \geq 0
\end{equation}
for all $u = \sum_{j,\alpha} u_{j,\alpha} \, \tfrac{\partial}{\partial z_j} \otimes e_\alpha \in T^{1,0}_xM \otimes E_x$, where $(z_1,\dotsc,z_n)$ are holomorphic coordinates of $M$ around $x$ and $\{e_\alpha\}_{\alpha}$ is an orthonormal basis of $E_x$.
The weaker condition of \defemph{Griffiths semipositivity} requires $R^E$ to satisfy \cref{eq:nakano_semi_positive} only on simple tensors, \ie for $u$ of the form $\tfrac{\partial}{\partial z_j} \otimes e$.
For more examples and properties of (Nakano or Griffiths) positive vector bundles, we refer to textbooks on complex geometry, for instance \cite{Ohsawa2015,Demailly2012}.

If $M$ is a \emph{complete} Hermitian manifold (without boundary, so $X=M$ in the above notation), then \defemph{$L^2$ Serre duality} (see \cite{Chakrabarti2012} for a detailed account) says that the Hodge star operator
\[ \hstar[E] \colon \Lambda^{\bullet,\bullet}T^*M \otimes E \to \Lambda^{n-\bullet,n-\bullet}T^*M \otimes E^*, \]
defined by $\langle u,v \rangle\,\volg = u \evwedge \hstar[E]v$, satisfies $\hstar[E] \circ \square^E = \square^{E^*} \circ \hstar[E]$.
Using this, one immediately obtains a result similar to \cref{spectral_percolation_specific}:

\begin{corx}
 \label{spectral_percolation_serre}
 Let $M$ be a K\"ahler manifold of $1$-bounded geometry, and let $E \to M$ be a Hermitian holomorphic vector bundle such that $E^*$ is $(n-q)$-Nakano lower semibounded.
 If $\square^E_{p,q+1}$ has discrete spectrum, then so does $\square^E_{p,q}$.
\end{corx}

\begin{proof}
 By $L^2$ Serre duality and our assumption, $\square^{E^*}_{n-p,n-q-1}$ has discrete spectrum.
 From \cref{spectral_percolation_specific}, it follows that $\square^{E^*}_{n-p,n-q}$ also has discrete spectrum, and another application of duality implies that $\square^E_{p,q}$ also has this property.
\end{proof}

We would like to remark that it was shown in \cite{Celik2014} that for pseudoconvex domains in $\Cplx^n$, the compactness of the minimal solution operators to the $\dbar$-equation also percolates up the Dolbeault complex, a property which is formally weaker than the corresponding statement for $\square_{p,\bullet}$. 

\subsection*{Generalized Schr\"odinger operators on line bundles}

Let $(M,g)$ be a (oriented) Riemannian manifold, and let $E \to M$ be a Hermitian vector bundle.
Then every connection $\nabla$ on $E$ and section $V \in \bsec(M,\en{E})$ defines an elliptic differential operator
\begin{equation}
 \label{eq:schrodinger_type_operator}
 H_{\nabla,V} \coloneqq \nabla^\fadj \nabla + V \colon \bsec(M,E) \to \bsec(M,E),
\end{equation}
called a \defemph{generalized Schr\"odinger operator}.
This is an operator of Laplace type and, conversely, any formally self-adjoint Laplace type operator is of this form, see \cite[Lemma~2.1]{Gilkey2008}.
In case $E$ is a line bundle, $\nabla$ is a metric connection, and $V$ is self-adjoint, operators of the form $H_{\nabla V}$ are sometimes called \defemph{magnetic Schr\"odinger operators} as they generalize the quantum Hamiltonian of a charged particle moving through an electromagnetic field.

We show that the appropriate generalization of a theorem of Iwatsuka \cite[Theorem~5.2]{Iwatsuka1986} continues to hold for Schr\"odinger operators acting on the sections of (possibly nontrivial) line bundles over manifolds of $1$-bounded geometry:

\begin{thmx}
 \label{necessary_condition_for_schrodinger_operator_discrete_spectrum}
 Let $L \to M$ be a Hermitian line bundle over a Riemannian manifold of $1$-bounded geometry, and let $H_{\nabla,V} \coloneqq \nabla^\fadj\nabla + V$ be a generalized Schr\"odinger operator for a metric connection $\nabla$ and self-adjoint morphism $V \colon L \to L$.  Assume that $H_{\nabla,V}$ has a lower semibounded self-adjoint extension with discrete spectrum.
 Then
 \[ \lim_{x \to \infty}\int_{B(x,r)} \big(|R^\nabla|^2 + |V|\big)\, d\mu_g = \infty \]
 for all $r>0$, where $R^\nabla$ is the curvature of $\nabla$.
\end{thmx}

\begin{rem}
 On $\R^n$, it is possible to characterize the discreteness of spectrum of operators of the form $-\Delta + V$ (\ie Schr\"odinger operators without magnetic field) by considering integrals of $|V|$ over sets which go to infinity, similarly to \cref{necessary_condition_for_schrodinger_operator_discrete_spectrum}.
 This is done in \cite{Molchanov1953}, and uses the concept of Wiener capacity of compact subsets of $\R^n$.
 There has also been progress to extend this to magnetic Schr\"odinger operators, see \cite{Kondratiev2004,Kondratiev2009}, but while some of those results are available on manifolds of bounded geometry, it is not clear what their geometric interpretation is, or if they can be generalized to the case of nontrivial line bundles.
\end{rem}

\Cref{necessary_condition_for_schrodinger_operator_discrete_spectrum} can be applied to the Dolbeault Laplacian on complete K\"ahler manifolds (again with some bounded geometry assumptions) on top (or bottom) degree forms with values in a Hermitian holomorphic line bundle, and we have the following result:

\begin{corx}
 \label{necessary_condition_square_discrete_spectrum}
 Let $L \to M$ be a Hermitian holomorphic line bundle over a K\"ahler manifold of $1$-bounded geometry, and let $p \in \{0,n\}$.
 Assume that one of the following conditions is satisfied:
 \begin{enumerate}
 \item $\square^L_{p,0}$ has discrete spectrum.
 \item $\square^L_{p,n}$ has discrete spectrum.
 \item For some $1 \leq q \leq n-1$, $L$ is $(q+1)$-Nakano lower semibounded and $\square^L_{p,q}$ has discrete spectrum.
 \item For some $1 \leq q \leq n-1$, $L^*$ is $(n-q+1)$-Nakano lower semibounded and $\square^L_{p,q}$ has discrete spectrum.  
 \end{enumerate}
 Then
 \begin{equation}
  \label{eq:necessary_condition_square_discrete_spectrum}
  \lim_{x \to \infty} \int_{B(x,r)} |R^L|^2\,d\mu_g = \infty
 \end{equation}
 for all $r>0$.
\end{corx}

Note that, as is the case for (Nakano) positivity of line bundles, the Nakano lower semiboundedness on $L^*$ in \cref{necessary_condition_square_discrete_spectrum} corresponds to Nakano \emph{upper} semiboundedness on $L$.

If $M=\Cplx^1$ and $L$ is the trivial line bundle with metric given by a weight $\varphi \colon \Cplx \to \R$, and if $\varphi$ is subharmonic and such that $\Delta\varphi$ defines a doubling measure, then the condition $\int_{B(x,1)} |R^L|\,d\mu_g \approx \int_{B(x,1)} \Delta \varphi\,d\mu_g \to \infty$ as $x \to \infty$ is known from \cite{Marzo2009} (or already \cite[Theorem~2.3]{Haslinger2007}, with slightly stronger assumptions) to be both necessary and sufficient for the discreteness of spectrum of $\square^L_{0,1}$.
A version of \cref{necessary_condition_square_discrete_spectrum} for the case $M=\Cplx^n$ appeared as joint work of the author with Friedrich Haslinger in \cite[Theorem~4.1]{Berger2017}.

\subsubsection*{Structure of the article}

In \cref{sec:differential_operators,sec:spec_elliptic_general}, we will develop some of the needed prerequisites on the essential spectrum of self-adjoint extensions of elliptic differential operators.
While the results therein are not fundamentally new, we believe that the presented generality merit their inclusion into this manuscript.
\Cref{sec:complex_geometry} provides the necessary concepts from complex and Hermitian geometry, with a focus on Weitzenb\"ock type formulas for the Dolbeault Laplacian which are needed in the proofs of \cref{spectral_percolation_specific,necessary_condition_square_discrete_spectrum}.
Finally, \cref{sec:proof_of_percolation,sec:schrodinger_operators} contain the proofs of the main results, and \cref{sec:bounded_geometry} provides background material on Riemannian manifolds of bounded geometry.

\ifams
\subsubsection*{Acknowledgments}
\else
\subsection*{Acknowledgments}
\fi
The results of this article are part of the author's doctoral research under the supervision of Prof.~Friedrich Haslinger.
The author wishes to thank Prof.~Haslinger for the many discussions on the subjects connected to this research, and Prof.~Siqi Fu for pointing out that the constants appearing in (the proof of) \cref{spectral_percolation_specific} can be enhanced by replacing an inequality employed in a previous version of this manuscript with the IMS localization formula~\cref{eq:ims_localization}.
\ifams\else
\thethanks{}
\fi

 \section{Preliminaries on (extensions of) differential operators}
\label{sec:differential_operators}

Let $M$ be a smooth manifold (always assumed to be second countable and, for simplicity, oriented), possibly with (smooth) boundary $\partial M$, and let $E \to M$ and $F \to M$ be smooth vector bundles.
For simplicity, we will always assume $M$ to be oriented.
We denote by $\bsec(M,E)$, $\bsecc(M,E)$, and $\bseccc(M,E)$ the spaces of smooth sections of $E$, of smooth sections of $E$ with compact support, and of smooth sections with compact support in the interior $\interior M$ of $M$, respectively.
Similarly, we have the function spaces $C^\infty(M)$, $C^\infty_c(M)$, and $C^\infty_\cc(M)$.

Let $D \colon \bsec(M,E) \to \bsec(M,F)$ be a (linear) differential operator, \ie a linear map that does not increase the support of sections.
For a cotangent vector $\gamma \in T^*_xM$, the \defemph{principal symbol} of $D$ at $\gamma$ is denoted by $\symb{D}(x,\gamma) \colon E_x \to F_x$, see for instance \cite[chapter~IV]{Palais1965}, and $D$ is called \defemph{elliptic} if $\symb{D}(x,\gamma)$ is a linear isomorphism for all $x \in M$ and $\gamma \in T^*_xM \setminus\{0\}$.
Suppose that $M$ is equipped with a Riemannian metric $g$, and that $E$ and $F$ carry Hermitian metrics, denoted by $\langle \arghere,\arghere \rangle$.
We then have the space $L^2(M,E)$, which is the Hilbert space completion of $\bseccc(M,E)$ with respect to
\begin{equation}
 \label{eq:ltwo_inner_product}
 \llangle s,t \rrangle \coloneqq \int_M \langle s,t \rangle\,d\mu_g.
\end{equation}
In~\cref{eq:ltwo_inner_product}, $\mu_g$ is the measure on $M$ induced by the metric and the orientation.
The \defemph{formal adjoint} to $D$ is the differential operator $D^\fadj \colon \bsec(M,F) \to \bsec(M,E)$ characterized by $\llangle Ds,t \rrangle = \llangle s,D^\fadj t \rrangle$ for all $s \in \bseccc(M,E)$ and $t \in \bsecc(M,F)$.
We use the sign convention for the principal symbol that makes $\symb{D^\fadj}(x,\gamma) = (-1)^k \symb{D}(x,\gamma)^*$, with $k$ the order of $D$.

A differential operator $D \colon \bsec(M,E) \to \bsec(M,E)$ on a Riemannian manifold is said to be of \defemph{Laplace type} if $\symb{D}(x,\gamma) = -|\gamma|^2\,\id{E_x}$ for all $x \in M$ and $\gamma \in T^*_xM$, and of \defemph{Dirac type} if $D^2$ is of Laplace type.
Dirac type operators acting on the sections of $E$ are in one-to-one correspondence with \emph{Clifford module structures} on $E$, \ie morphisms $c \colon T^*M \otimes E \to E$ with $c(\gamma) c(\beta) + c(\beta) c(\gamma) = -2\langle \gamma,\beta\rangle\id{E}$, where we write $c(\gamma) \coloneqq c(\gamma \otimes \arghere)$.
A tuple $(E,M,c,\nabla)$ with $E \to M$ a Hermitian vector bundle over a Riemannian manifold, $c$ a Clifford module structure on $E$, and $\nabla$ a metric connection, is called a \defemph{Dirac bundle} in the sense of \cite{Lawson1989} if $\nabla c = 0$ and $c(\gamma,\arghere)$ is skew-Hermitian for all $\gamma \in T^*M$.
The associated Dirac operator $D \coloneqq c \circ \nabla$ is then formally self-adjoint.

Given a connection $\nabla \colon \bsec(M,E) \to \Omega^1(M,E)$ on $E$, we denote by $d^\nabla \colon \Omega^\bullet(M,E) \to \Omega^{\bullet+1}(M,E)$ the \defemph{exterior covariant derivative} associated to $\nabla$, which satisfies $d^\nabla s = \nabla s$ for $s \in \bsec(M,E)$ and
\begin{equation}
 \label{eq:exterior_covariant_derivative}
 d^\nabla(\alpha \wedge u) = d\alpha \wedge u + (-1)^k \alpha \wedge d^\nabla u
\end{equation}
for $\alpha \in \Omega^k(M)$ and $u \in \Omega(M,E)$.
Choosing a torsion free connection on $TM$, this is $d^\nabla = \extpsymb \circ \widetilde \nabla$, with $\extpsymb$ the exterior multiplication map, and $\widetilde \nabla$ the connection induced on $\Lambda T^*M \otimes E$, see \cite[Theorem~12.56]{Lee2009}.
The curvature of $\nabla$ then satisfies $R^{\nabla} \evwedge u = d^\nabla d^\nabla u$ for $u \in \Omega(M,E)$, where the wedge product is combined with the evaluation map $\en{E} \otimes E \to E$.
If $M$ is a complex manifold, then $d^\nabla_{1,0}$ denotes the $(1,0)$-part of $d^\nabla$, defined as $\Pi_{p+1,q} \circ d^\nabla$ on $\Omega^{p,q}(M,E)$, with $\Pi_{p+1,q}$ the projection onto $\Lambda^{p+1,q}T^*M \otimes E$.
Similarly, the $(0,1)$-part $d^\nabla_{0,1}$ is defined.
If $E$ is Hermitian, then the \defemph{Chern connection} on $E$ is the unique metric connection $\nabla$ with $d^{\nabla}_{0,1} = \dbar^E$, see for instance \cite[Theorem~2.1]{Wells2008}.

\subsection{The IMS localization formula}
\label{sec:ims_localization}

Let $E \to M$ be a Hermitian vector bundle over a Riemannian manifold.
For a formally self-adjoint second order differential operator $H \colon \bsec(M,E) \to \bsec(M,E)$, the \defemph{IMS localization formula}\footnote{  Usually, this terminology is only applied to formula~\cref{eq:ims_localization} for $H$ a Laplace type operator, see \cite[Theorem~3.2]{Cycon1987}, in which case $\symb{D}(d\varphi_k)s = -|d\varphi_k|^2 s$.
  According to \cite{Cycon1987}, the acronym stands for Ismagilov \cite{Ismagilov1961}, Morgan \cite{Morgan1979}, Morgan--Simon \cite{Morgan1980}, and I.M.~Sigal \cite{Sigal1982}.
}
reads
\begin{equation}
 \label{eq:ims_localization}
 \llangle Hs, s \rrangle
 = \sum_{k=1}^\infty \big(\llangle H(\varphi_k s), \varphi_k s \rrangle + \llangle \symb{H}(d\varphi_k) s, s \rrangle\big)
\end{equation}
for $s \in \bseccc(M,E)$ and $\varphi_k \in C^\infty_c(M,[0,1])$ having the property that $(\varphi_k^2)_{k=1}^\infty$ is a partition of unity for $M$.
It is easily obtained by invoking the definition of the principal symbol of $H$, namely
\[
 \symb{H}(d\varphi_k)s = \frac{1}{2}\big[[H,\varphi_k],\varphi_k\big]s = \frac{1}{2}\big(H(\varphi_k^2s) + \varphi_k^2 Hs\big) - \varphi_k H(\varphi_k s),
\]
then integrating and summing over $k$.
A similar localization formula holds for first order operators $D \colon \bsec(M,E) \to \bsec(M,F)$:
If $s \in \bsecc(M,E)$, then
\begin{align*}
  \|Ds\|^2
  & = \sum_{k=1}^\infty \re \llangle D s, D (\varphi_k^2 s) \rrangle \\
  & = \sum_{k=1}^\infty \re \llangle \varphi_k D s, D(\varphi_ks) + \symb{D}(d\varphi_k)s \rrangle \\
  & = \sum_{k=1}^\infty \re \llangle D(\varphi_k s) - \symb{D}(d\varphi_k)s, D(\varphi_k s) + \symb{D}(d\varphi_k)s \rrangle \\
  & = \sum_{k=1}^\infty \big(\|D(\varphi_k s)\|^2 - \|\symb{D}(d\varphi_k) s\|^2\big). \numberthis\label{eq:ims_localization_first_order}
\end{align*}

\subsection{Strong and weak extensions of differential operators}
\label{sec:strong_weak_extensions}

If $D \colon \bsec(M,E) \to \bsec(M,F)$ is a differential operator, then it makes sense to ask whether the linear map
\[ D_\cc \coloneqq D|_{\bseccc(M,E)} \colon \bseccc(M,E) \to L^2(M,F) \]
extends to a closed operator from $L^2(M,E)$ to $L^2(M,F)$. Let $D_\cc^*$ denote the Hilbert space adjoint of $D_\cc$.
The definition of $D^\fadj$ implies that $(D^\fadj)_\cc \subseteq D_\cc^*$, meaning $\graph{(D^\fadj)_\cc} \subseteq \graph{D_\cc^*}$, hence $D_\cc$ is closable since its adjoint is densely defined.
To save on notational clutter, we shall say that a linear operator $A \colon \dom{A} \subseteq L^2(M,E) \to L^2(M,F)$ is an \defemph{extension of $D$} (or ``\defemph{$A$ extends $D$}'') if $D_\cc \subseteq A$, and a \defemph{closed extension of $D$} if, in addition, $A$ is closed.

\begin{defn}
 The \defemph{strong extension} (or \defemph{minimal closed extension}) of $D$, denoted by $D_\minn$, is the closure of $D_\cc \colon \bseccc(M,E) \subseteq L^2(M,E) \to L^2(M,F)$, and the \defemph{weak extension} (or \defemph{maximal closed extension}) of $D$ is $D_\maxx \coloneqq (D^\fadj)_\cc^*$.
\end{defn}

Since $D_\cc = (D^{\fadj\fadj})_\cc \subseteq (D^\fadj)_\cc^*$, the operator $D_\maxx$ really is an extension of $D$.
Both the strong and weak extensions of $D$ are closed, and hence $D_\minn \subseteq D_\maxx$.
It holds that $D_\maxx = ((D^\fadj)_\minn)^*$, since a densely defined operator and its closure have the same adjoint.
This immediately implies
\begin{equation}
 \label{eq:weak_and_strong_extension_of_formal_adjoint}
 (D^\fadj)_\maxx = (D_\minn)^* \quad\text{and}\quad (D^\fadj)_\minn = (D_\maxx)^*.
\end{equation}
By the definition of the formal adjoint of $D^\fadj$, we have $\llangle (D^\fadj)_\cc t,s \rrangle = \llangle t,Ds \rrangle$ for $t \in \bseccc(M,F)$ and $s \in \bsecc(M,E)$, thus $\bsecc(M,E) \subseteq \dom{(D^\fadj)_\cc^*} = \dom{D_\maxx}$ and $D_\maxx|_{\bsecc(M,E)} = D|_{\bsecc(M,E)}$.
In particular, every extension $A$ of $D$ with $A \subseteq D_\maxx$ satisfies
\begin{equation}
 \label{eq:extension_of_diff_op_restrictions_to_diff_op}
 A|_{\dom{A}\cap \bsecc(M,E)} = D|_{\dom{A} \cap \bsecc(M,E)}.
\end{equation}
As its name suggests, the weak extension admits a description in terms of the distributional action of $D$:
It is easy to see that
\[ \dom{D_\maxx} = \big\{s \in L^2(M,E) : Ds \in L^2(M,F) \text{ in the sense of distributions}\big\} \]
and $D_\maxx s = Ds$ for $s \in \dom{D_\maxx}$, where $Ds$ is the distributional derivative.

\begin{rem}
 \label{rem:symmetric_extensions_of_diff_ops}
 It is clear that $D_\minn$ is the smallest extension of $D_\cc$ to a closed operator from $L^2(M,E)$ to $L^2(M,F)$.
 The weak extension $D_\maxx$ is maximal in the sense that it is the largest extension of $D$ whose adjoint extends $D^\fadj$, \ie contains $\bseccc(M,F)$ in its domain, see \cite{Grieser2002}.
 If $A$ is a \emph{symmetric} extension of a (necessarily formally self-adjoint) differential operator $D \colon \bsec(M,E) \to \bsec(M,E)$, then $D_\cc \subseteq A \subseteq A^*$, so $\bseccc(M,E) \subseteq \dom{A} \subseteq \dom{A^*}$.
 Thus, $A$ is a restriction of $D_\maxx$.
 This comes as no surprise, since $D_\maxx = (D^\fadj)_\maxx = (D_\minn)^* = D_\cc^*$ and all symmetric extensions of a symmetric operator on a Hilbert space are restrictions of its adjoint.
\end{rem}

An application of the usual interior regularity estimates for elliptic operators (see \cite[Theorem~5.11.1]{Taylor2011}) is the following statement about the regularity of sections in the domain of the weak extension of an elliptic operator, a proof of which can be found, for instance, in \cite[Proposition~2.1]{Bei2017}:

\begin{lem}
 \label{domain_of_extension_in_local_sobolev_space}
 Let $D \colon \bsec(M,E) \to \bsec(M,F)$ be a $k$\textsuperscript{th} order elliptic differential operator, and suppose $A$ is an extension of $D$ with $D_\minn \subseteq A \subseteq D_\maxx$. \footnote{In particular, this is true for self-adjoint $A$, see \cref{rem:symmetric_extensions_of_diff_ops}.}
 Then $\dom{A} \subseteq H^k_\loc(M,E)$, and if $A$ is closed, then it has a core consisting of sections which are smooth on $\interior M$.
\end{lem}

\begin{lem}
 \label{multiplication_by_compact_support_function_maps_to_strong_domain}
 Let $D \colon \bsec(M,E) \to \bsec(M,F)$ be an elliptic differential operator and $\varphi \in C^\infty_\cc(M)$.
 Then the operator of multiplication by $\varphi$ maps $\dom{D_\maxx}$ to $\dom{D_\minn}$.
\end{lem}

For first order operators that are not necessary elliptic, \cref{multiplication_by_compact_support_function_maps_to_strong_domain} can be found in \cite{Grieser2002}.
This is basically Friedrich's lemma.
The present version for elliptic operators of higher order makes use of the elliptic estimates:

\begin{proof}[Proof of \cref{multiplication_by_compact_support_function_maps_to_strong_domain}]
 By \cref{domain_of_extension_in_local_sobolev_space}, $D_\maxx$ has a core consisting of sections which are smooth on $\interior M$.
 Let $s \in \dom{D_\maxx}$, and choose $s_k \in \bsec(\interior M,E) \cap \dom{D_\maxx}$ with $s_k \to s$ in $\dom{D_\maxx}$.
 Then $\varphi s_k \in \bseccc(M,E) \subseteq \dom{D_\minn}$ and $\varphi s_k \to \varphi s$ in $L^2(M,E)$.
 Moreover,
 \[ \|D_\minn(\varphi s_k) - D_\minn(\varphi s_j)\| \leq \|\varphi (Ds_k - Ds_j)\| + \|[D,\varphi](s_k - s_j)\|. \]
 Choose relatively compact open subsets $U \subset\subset V \subseteq \interior M$ such that $\supp(\varphi) \subseteq U$.
 If $d$ is the order of $D$, then $[D,\varphi]$ has order $d-1$ and vanishes outside of $\supp(\varphi)$, hence there is a constant $C > 0$ such that $\|[D,\varphi]u\| \leq C \|u\|_{H^{d-1}(U,E)}$ for all $u \in \bseccc(U,E)$.
 By the elliptic estimates, we therefore have
 \[ \|D_\minn(\varphi s_k) - D_\minn(\varphi s_j)\| \leq \|\varphi\|_{L^\infty} \|Ds_k - Ds_j\| + \widetilde C \big(\|Ds_k - Ds_j\| + \|s_k-s_j\|\big) \]
 for some constant $\widetilde C > 0$ and all $j,k \geq 1$.
 We conclude that $(\varphi s_k)_{k \in \N}$ is Cauchy in $\dom{D_\minn}$, hence convergent, and the limit must agree with $\varphi s$ by the convergence in $L^2(M,E)$.
\end{proof}

\subsection{Density of sections with compact support}
\label{sec:essential_self-adjointness}

In this \namecref{sec:essential_self-adjointness}, we want to study whether sections with compact support are dense in $\dom{A}$ for the graph norm $s \mapsto (\|s\|^2 + \|As\|^2)^{1/2}$, with $A$ a closed extension of a differential operator.
Put differently: does $A$ have a core consisting of sections with compact support?
The results will be for \emph{first order} differential operators. 

By a \defemph{complete Riemannian manifold} $(M,g)$ we mean a connected manifold $M$, possibly with boundary, together with Riemannian metric $g$ such that the Riemannian distance $d_g$ turns $M$ into a complete metric space.
A generalization of the theorem of Hopf--Rinow says that $(M,d_g)$ is complete if and only if its compact subsets are exactly the closed and bounded ones, see \cite[p.~9]{Gromov1999}.
The following \namecref{characterization_of_complete_manifolds} is a standard characterization of complete Riemannian manifolds, and is also true for manifolds with boundary.
It can easily be proved by adapting the arguments from \cite[Lemma~12.1]{Demailly2002}.

\begin{lem}
 \label{characterization_of_complete_manifolds}
 Let $(M,g)$ be a connected Riemannian manifold, possibly with boundary.
 Then $(M,g)$ is complete if and only if
  there is a sequence $(\chi_k)_{k \in \N}$ of functions in $C^\infty_c(M,[0,1])$ with $(\chi_{k+1})|_{\supp(\chi_k)} = 1$ and $|d\chi_k|\leq 2^{-k}$ for all $k \in \N$, and such that $(\supp(\chi_k))_{k \in \N}$ is a compact exhaustion of $M$.
\end{lem}

The central condition on $A$ will be the validity of the following Leibniz rule.
For $s \in \bseccc(M,E)$, the equation~\cref{eq:leibniz_rule_for_extension} is just the definition of the principal symbol of $D$.

\begin{defn}
 \label{def:leibniz_rule_for_extension}
 Let $D \colon \Gamma(M,E) \to \Gamma(M,F)$ be a first order differential operator.
 We say that an extension $A$ of $D$ \defemph{satisfies the Leibniz rule (with respect to $C^\infty_c(M)$)} if $fs \in \dom{A}$ and
 \begin{equation}
  \label{eq:leibniz_rule_for_extension}
  A(fs) = fAs + \symb{D}(df)s
 \end{equation}
 for all $s \in \dom{A}$ and $f \in C^\infty_c(M)$.
 (Note that the support of $f$ may intersect the boundary.)
\end{defn} 

\begin{thm}
 \label{bounded_symbol_compact_supported_sections_dense}
 Let $A$ be a closed extension of a first order differential operator $D$ satisfying the Leibniz rule \cref{eq:leibniz_rule_for_extension}.
 Suppose that $(M,g)$ is complete and that the principal symbol of $D$ satisfies
 \begin{equation}
  \label{eq:bounded_symbol_compact_supported_sections_dense}
  |\symb{D}(\gamma)| \leq C|\gamma|
 \end{equation}
 for some constant $C > 0$ and all $\gamma \in T^*M$.
 If $W \subseteq \dom{A}$ is a core for $A$,
 then $\{\varphi s : \varphi \in C^\infty_c(M),\, s \in W\}$ is also a core for $A$.
 In particular, the compactly supported elements are dense in $\dom{A}$. \end{thm}

\begin{proof}
 We slightly modify the proof of \cite[Theorem~3.3]{Baer2012}, where the statement is shown for $D_\maxx$, \cf \cref{leibniz_formula_for_weak_and_strong_extensions} below.
 Let $s \in \dom{A}$.
 Since $W$ is a core for $A$, we find $s_k \in W$ with $s_k \to s$ in $\dom{A}$.
 By the completeness of $(M,g)$, there exists a sequence $(\chi_k)_{k \in \N}$ of functions in $C^\infty_c(M,[0,1])$ with $(\chi_{k+1})|_{\supp(\chi_k)} = 1$, and $|d\chi_k| \leq 2^{-k}$ for all $k \in \N$, see \cref{characterization_of_complete_manifolds}.
 Then $\chi_k s_k$ has compact support and is an element of $\dom{A}$ by assumption.
 By the dominated convergence theorem, $\|\chi_k s - s\| \to 0$ and $\|\chi_k As - As\| \to 0$.
 It follows that $\chi_k s_k \to s$ in $L^2(M,E)$, and
 \begin{multline*}
  \|A(\chi_k s_k) - As \| \\
  \begin{aligned}
   & \leq \|A(\chi_k s_k) - A(\chi_k s)\| + \|A(\chi_k s) - As)\| \\
   & \leq \|\chi_k A(s_k - s)\| + \|\symb{D}(d\chi_k)(s_k - s)\| + \|\chi_k As - As\| + \|\symb{D}(d\chi_k)s\| \\
   & \leq \|As_k - As\| + \frac{C}{2^k}\|s_k - s\| + \|\chi_k As - As\| + \frac{C}{2^k} \|s\|
  \end{aligned}
 \end{multline*}
 also converges to zero as $k \to \infty$.
 Therefore, $\chi_k s_k \to s$ in $\dom{A}$, so $\{\varphi s : \varphi \in C^\infty_c(M), s \in W\} \subseteq \{s \in \dom{A} : \supp(s) \text{ compact}\}$ is a core for $A$.
\end{proof}

\begin{rem}
 A more sophisticated condition is given in \cite[Theorem~1.2]{Baer2012}:
 if $M$ is a connected Riemannian manifold which admits a complete Riemannian metric $h$ such that
 \[ |\symb{D}(\gamma)| \leq C(\dist[d_h]{x}{\partial M}) |\gamma|_h \]
 for all $x\in M$ and $\gamma \in T^*_xM$, with $C \colon [0,\infty) \to \R$ a positive, continuous, and monotonically increasing function satisfying
 $\int_0^\infty \tfrac{1}{C(r)}\,dr = \infty$,
 then compactly supported elements of $\dom{D_\maxx}$ are a core for $D_\maxx$.
 After a conformal change of metric, this case is reduced to \cref{eq:bounded_symbol_compact_supported_sections_dense}.
\end{rem}

\begin{example}
 \label{ex:dirac_operator_squared_leibniz_rule}
 If $D$ is a formally self-adjoint differential operator of Dirac type, then
 \[
  |\symb{D}(\gamma)|^2 = |\symb{D}(\gamma)^*\symb{D}(\gamma)| = |\symb{D^2}(\gamma)| = |\gamma|^2
 \]
 for all $\gamma \in T^*M$, hence \cref{eq:bounded_symbol_compact_supported_sections_dense} is satisfied.
\end{example}

Of course, the value of \cref{bounded_symbol_compact_supported_sections_dense} depends on the number of extensions of $D$ for which the Leibniz rule can be established.
The next \namecref{leibniz_formula_for_closure_adjoint} gives us something to work with:

\begin{prop}
 \label{leibniz_formula_for_closure_adjoint}
 Let $A$ be an extension of a first order differential operators $D$ with $D_\cc \subseteq A \subseteq D_\maxx$ and satisfying the Leibniz rule~\cref{eq:leibniz_rule_for_extension}.
 Then its closure $\ol A$ satisfies the Leibniz rule, and the extension $A^*$ of $D^\fadj$ also has this property, \ie $fs \in \dom{A^*}$ and
 \[ A^*(fs) = f A^*s + \symb{D^\fadj}(df)s \]
 for all $s \in \dom{A^*}$ and $f \in C^\infty_c(M)$.
\end{prop}

\begin{proof}
 Let $s \in \dom{A^*}$, $f \in C^\infty_c(M,\R)$, and $t \in \dom{A}$.
 Then $ft \in \dom{A}$ and $A(ft) = fAt + \symb{D}(df)t$ by assumption, and
 \begin{multline*}
  \llangle fs, At \rrangle
  = \llangle s, f At \rrangle
  = \llangle s, A(f t) - \symb{D}(df) t\rrangle = \\
  = \llangle A^* s, f t \rrangle + \llangle \symb{D^\fadj}(df)s, t \rrangle
  = \llangle f A^* s + \symb{D^\fadj}(df)s, t \rrangle.
 \end{multline*}
 This implies $fs \in \dom{A^*}$ and $A^*(fs) = fA^* s + \symb{D^\fadj}(df)s$.
 The closure of $A$ is given by $\ol A = A^{**}$, so the claim for $\ol A$ follows immediately.
\end{proof}

\begin{cor}
 \label{leibniz_formula_for_weak_and_strong_extensions}
 Let $D$ be a first order differential operator.
 Then $D_\minn$ and $D_\maxx$ satisfy the Leibniz rule~\cref{eq:leibniz_rule_for_extension}.
\end{cor}

\begin{proof}
 Clearly, $D_\cc$ and $(D^\fadj)_\cc$ both satisfy the Leibniz rule, so \cref{leibniz_formula_for_closure_adjoint} implies that $D_\minn = \ol{D_\cc}$ and $D_\maxx = ((D^\fadj)_\cc)^*$ also have this property.
\end{proof}

\begin{rem}
 \label{rem:leibniz_rule_for_lipschitz_functions}
 The proof of \cref{leibniz_formula_for_closure_adjoint} also works if we replace $C^\infty_c(M)$ by the space of bounded smooth functions $f \colon M \to \R$ such that $x \mapsto |\symb{D}(x,df(x))|$ is bounded on $M$.
 If $D$ satisfies the symbol bound \cref{eq:bounded_symbol_compact_supported_sections_dense}, then bounded smooth Lipschitz functions have this property.
 In particular, $D_\maxx$ and $D_\minn$ satisfy the Leibniz rule with respect to these functions.
\end{rem}

\begin{example}
 \label{ex:hilbert_complexes_leibniz_rule}
 Let $(E_\bullet,d)$ be an \defemph{elliptic complex} of first order differential operators, \ie the associated symbol sequence is exact, meaning $\img{\symb{d_i}(\gamma)} = \ker(\symb{d_{i+1}}(\gamma))$ for all nonzero $\gamma \in T^*M$, see for instance \cite{Taylor2011a}.
 The operator $d_\maxx + d_\maxx^* = d_\maxx + (d^\fadj)_\minn$ is a self-adjoint extension of $d+d^\fadj$, see the arguments in \cite[Proposition~2.3]{Goldshtein2011}, with domain $\dom{d_\maxx} \cap \dom{d_\maxx^*} = \dom{d_\maxx} \cap \dom{(d^\fadj)_\minn}$.
 It is easy to see that the differential operator $d+d^\fadj$ is also of first order, with $\symb{d+d^\fadj} = \symb{d} + \symb{d^\fadj}$.
 Together with \cref{leibniz_formula_for_weak_and_strong_extensions}, this immediately gives that $d_\maxx + d_\maxx^*$ satisfies the Leibniz rule.
 Similarly, one shows that this is also true for $d_\minn + d_\minn^*$.
\end{example}

 \section{The essential spectrum of self-adjoint elliptic differential operators}
\label{sec:spec_elliptic_general}

In this \namecref{sec:spec_elliptic_general} we consider (nonnegative) self-adjoint extensions $A$ of general elliptic differential operators on a Riemannian manifold $M$, possibly having a boundary.
\Cref{sec:decomposition_principle} will first set up the notation used throughout this \namecref{sec:spec_elliptic_general}, the highlight of which is the \defemph{decomposition principle}, which states that one can restrict $A$ to complements of compact subsets of $\interior M$ without changing the essential spectrum.
In \cref{sec:bottom_of_essspec}, the bottom of the essential spectrum of such operators is considered.
One of the key results there is \cref{limit_of_infimum_of_spectrum}, a generalization of a theorem of Persson~\cite{Persson1960}, and it states that $\inf\essspec{A}$ is the limit of the net $K \mapsto \inf\spec{A_{M\setminus K}}$, where $K$ runs through the compact subsets of $\interior M$, directed by inclusion.
The results in this \namecref{sec:spec_elliptic_general} are not fundamentally new, but we have taken care to keep them as general as possible.
For instance, we shall not make the often used assumption for $A$ to have a core of smooth sections with compact support (although this will be satisfied in our applications).

\subsection{The decomposition principle}
\label{sec:decomposition_principle}

Let $(M,g)$ be a Riemannian manifold with (possibly empty) boundary $\partial M$, and let $E \to M$ be a (complex) Hermitian vector bundle.
Suppose $D \colon \bsec(M,E) \to \bsec(M,E)$ is a formally self-adjoint differential operator of order at least one, and let
\[ A \colon \dom{A} \subseteq L^2(M,E) \to L^2(M,E) \]
be a lower semibounded self-adjoint extension of $D$, by which we mean $D_\cc \subseteq A$.
We denote by $Q_A$ the quadratic form associated to $A$, see for instance \cite{Schmuedgen2012}.
The space $\dom{A}$ is complete under the norm $s \mapsto (\|s\|^2 + \|As\|^2)^{1/2}$, and $\dom{Q_A}$ is a Hilbert space when equipped with $s \mapsto (\|s\|^2 + Q_A(s,s))^{1/2}$.
In order to formulate the results of the following sections, it will be convenient to restrict $A$ to open subsets of $M$:

\begin{lem}
 \label{localized_form_is_closable}
 Let $U \subseteq M$ be an open subset.
 Then the quadratic form $\widetilde Q_{A,U}$ on $L^2(U,E)$ with $\dom{\widetilde Q_{A,U}} \coloneqq \{s|_U : s \in \dom{Q_A} \text{ and } \supp(s) \subseteq U\}$ and $\widetilde Q_{A,U}(s,s) \coloneqq Q_A(s_0,s_0)$ for $s \in \dom{\widetilde Q_{A,U}}$ is closable, where $s_0 \in L^2(M,E)$ denotes the extension of $s$ by zero and $\supp(s)$ is to be taken in the measure senses.
\end{lem}

\begin{proof}
 We need to show that if $u_k \in \dom{\widetilde Q_{A,U}}$ is a sequence with $u_k \to 0$ in $L^2(U,E)$ and such that for every $\varepsilon > 0$ there is $N \in \N$ with $|\widetilde Q_{A,U}(u_k - u_j,u_k - u_j)| \leq \varepsilon$ for $j,k \geq N$, then also $\widetilde Q_{A,U}(u_k,u_k) \to 0$ as $k \to \infty$, see \cite[Proposition~10.3]{Schmuedgen2012}.
 These assumptions on $u_k$ imply that $((u_k)_0)_{k \in \N}$ is Cauchy in $\dom{Q_A}$.
 Since $Q_A$ is closed, there is $t \in \dom{Q_A}$ with $(u_k)_0 \to t$ in $\dom{Q_A}$, and as also $(u_k)_0 \to 0$ in $L^2(M,E)$ by assumption, we have $t = 0$.
 Now
 \[ \widetilde Q_{A,U}(u_k,u_k) = Q_A((u_k)_0,(u_k)_0) \to Q_A(t,t) = 0 \]
 as $k \to \infty$, so $\widetilde Q_{A,U}$ is closable.
\end{proof}

\begin{defn}
 \label{def:localized_differential_operator}
 For $U \subseteq M$ an open subset, we define the quadratic form $Q_{A,U}$ as the closure of the quadratic form $\widetilde Q_{A,U}$ from \cref{localized_form_is_closable}.
 The self-adjoint operator associated to $Q_{A,U}$ is denoted by $A_U$.
\end{defn}

Note that the open subset $U$ in \cref{def:localized_differential_operator} is allowed to intersect $\partial M$.
We think of $A_U$ as being obtained by putting Dirichlet boundary conditions on $\partial U \setminus \partial M$, and keeping the original boundary conditions on $\partial M \cap U$.
Since $\dom{\widetilde Q_{A,U}}$ is dense in $\dom{Q_{A,U}}$, the operator $A_U$ is given by
\begin{multline}
 \label{eq:domain_of_localized_operator}
 \dom{A_U} = \big\{s \in \dom{Q_{A,U}} : \text{there is } u_s \in L^2(U,E) \text{ such that } \\
 Q_{A,U}(s,t)
 = \llangle u_s,t \rrangle_{L^2(U,E)}
 \text{ for all } t \in \dom{Q_{A}} \text{ with } \supp(t) \subseteq U\big\},
\end{multline}
and $A_U s \coloneqq u_s$ for $s \in \dom{A_U}$.
It is not hard to see that $\{s|_U : s \in \dom{A} \text{ and } \supp(s) \subseteq U\}$ is contained in $\dom{A_U}$ and that $A_U(s|_U) = (As)|_U$ holds for all $s$ in this set.

\begin{lem}
 \label{extension_by_zero_belongs_to_domain}
 Let $U,V \subseteq M$ be open subsets with $U \subseteq V$.
 If $u \in \dom{Q_{A,U}}$, then $u_0 \in L^2(V,E)$ belongs to $\dom{Q_{A,V}}$, and $Q_{A,U}(u,v) = Q_{A,V}(u_0,v_0)$ for all $u,v \in \dom{Q_{A,U}}$.
 In particular, $\inf\spec{A_U} \geq \inf\spec{A_V}$.
\end{lem}

\begin{proof}
 Let $s_k \in \dom{Q_A}$ be a sequence with $\supp(s_k) \subseteq U$ and $s_k|_U \to u$ in $\dom{Q_{A,U}}$.
 Then the definition of $Q_{A,U}$ implies that $k \mapsto s_k$ is Cauchy in $\dom{Q_A}$, hence convergent to some $s_\infty \in \dom{Q_A}$.
 Moreover,
 \[ Q_{A,U}(u,u) = \lim_{k \to \infty} \widetilde Q_{A,U}(s_k|_U,s_k|_U) = \lim_{k \to \infty} Q_A(s_k,s_k) = Q_A(s_\infty,s_\infty). \]
 Similarly, since $\supp(s_k) \subseteq V$, we have $s_k|_V \to u_0$ in $L^2(V,E)$ and $k \mapsto s_k|_V$ is Cauchy in $\dom{Q_{A,V}}$, hence convergent to $u_0$ in $\dom{Q_{A,V}}$.
 Thus,
 \[ Q_{A,V}(u_0,u_0) = \lim_{k \to \infty} \widetilde Q_{A,V}(s_k|_V,s_k|_V) = \lim_{k \to \infty} Q_A(s_k,s_k) = Q_A(s_\infty,s_\infty) = Q_{A,U}(u,u). \]
 By the polarization identity, we get equality also away from the diagonal.

 The inequality about the bottom of the spectra follows from the fact that $\inf\spec{A_U}$ is the largest lower bound of $Q_{A,U}$, since we have
 \[ Q_{A,U}(s,s) = Q_{A,V}(s_0,s_0) \geq (\inf\spec{A_V})\|s_0\|^2 = (\inf\spec{A_V})\|s\|^2 \]
 for all $s \in \dom{Q_{A,U}}$, hence $\inf\spec{A_U} \geq \inf\spec{A_V}$.
\end{proof}

\begin{thm}[Decomposition principle]
 \label{decomposition_principle}
 Let $A$ be a lower semibounded self-adjoint extension of an elliptic differential operator as above.
 Then
 \[ \essspec{A} = \essspec{A_{M\setminus K}} \]
 for all compact $K \subseteq \interior{M}$. \end{thm}

A proof can be found in \cite[Proposition~1]{Baer2000}.
It works by using the characterization of the essential spectrum through \defemph{singular Weyl sequences}:
Namely, $\lambda \in \essspec{A}$ if and only if there is a sequence $s_k \in D_0$, where $D_0$ is any core of $A$, with $s_k \to 0$ weakly, $\liminf_{k \to \infty} \|s_k\| > 0$, and $(A-\lambda)s_k \to 0$.
Elliptic estimates are used to obtain a singular Weyl sequence for $(A_{M\setminus K},\lambda)$ from a singular Weyl sequence for $(A,\lambda)$ by truncation with a cutoff function.
Other sources with similar statements or for certain classes of operators include \cite[Proposition~4.9]{Eichhorn1988}, \cite[Proposition~1.4]{Eichhorn2007}, and \cite[Proposition~3.2.4]{Ma2007}.
A minor difference between the decomposition principle of B\"ar from \cite{Baer2000} and \cref{decomposition_principle} is that we do not assume that $\bsecc(M,E) \cap \dom{A}$ is a core for $A$.
This is not an issue, since $\bsec(\interior M,E) \cap \dom{A}$ is a core of $A$ by \cref{domain_of_extension_in_local_sobolev_space}, and the argument from \cite{Baer2000} can then be carried out in exactly the same way, without using the fact that the $s_k$ have compact support.
One merely needs that multiplication by a cutoff function which is constant outside a compact subset preserves $\dom{A}$, but this is immediate from \cref{multiplication_by_compact_support_function_maps_to_strong_domain,rem:symmetric_extensions_of_diff_ops}.

\subsection{The bottom of the essential spectrum}
\label{sec:bottom_of_essspec}

In this section, we wish to study the bottom of the essential spectrum of a nonnegative self-adjoint extension $A$ of an elliptic differential operator $D \colon \bsec(M,E) \to \bsec(M,E)$.
Of course, the results also apply to lower semibounded operators after straightforward modifications.

Apart from the decomposition principle in \cref{decomposition_principle}, one of the main tools used in the rest of this section will be the following simple and well-known property of compact subsets of $L^p(M,E)$, the proof of which can be found, for instance, in \cite[Theorem~2.32]{Adams2003} or \cite[Theorem~2.5]{Ruppenthal2011a}:

\begin{lem}\label{compact_subsets_in_L2_spaces}
 Let $1 \leq p < \infty$.
 Suppose that $E$ is a Hermitian vector bundle over a Riemannian manifold $M$, and let $B \subseteq L^p(M,E)$ be a totally bounded (equivalently: relatively compact) subset.
 Then for every $\varepsilon > 0$ there exists a compact subset $K \subseteq \interior M$ such that
 \[
 \int_{M\setminus K} |s|^p\,d\mu_g \leq \varepsilon
 \]
 for all $s \in B$.
\end{lem}

   In particular, if $T \colon X \to L^p(M,E)$ is a compact linear operator, with $(X,\|\arghere\|_X)$ a Banach space, then
   \cref{compact_subsets_in_L2_spaces} implies that there exists, for every $\varepsilon > 0$, a compact subset $K \subseteq \interior M$ such that
   \begin{equation}
    \label{eq:compact_subsets_in_L2_spaces_compact_operator}
    \int_{M\setminus K} |Tx|^p\,d\mu_g \leq \varepsilon \|x\|_X^p
   \end{equation}
   for all $x \in X$.
We are now ready to show our main \namecref{below_essspec_function} for this section:

\begin{lem}
 \label{below_essspec_function}
 Let $A$ be a nonnegative self-adjoint operator\footnote{Note that $A$ need not be an extension of a differential operator.} on $L^2(M,E)$.
 For every $0 < \lambda < \inf\essspec{A}$ and $\varepsilon > 0$, there exists a compact subset $K \subseteq \interior M$ such that
 \begin{equation*}
  \label{eq:below_essspec_function}
  Q_A(s,s)
  \geq \int_M (\lambda \chi_{M\setminus K} - \varepsilon \chi_K)|s|^2\,d\mu_g
  = \lambda \int_{M\setminus K} |s|^2\,d\mu_g - \varepsilon \int_K |s|^2\,d\mu_g
 \end{equation*}
 for all $s \in \dom{Q_A}$, where $\chi_{M\setminus K}$ and $\chi_K$ are the characteristic functions.
\end{lem}

\begin{proof}
 Denote by $P_A$ the spectral measure associated to $A$, and let
 \[ 0 < \delta < \min\{\inf\essspec{A} - \lambda,\varepsilon/2\}. \]
 Put $P_0 \coloneqq P_A([0,\lambda + \delta])$.
 Then the inclusion $\img{P_0} \cap \dom{Q_A} \hookrightarrow L^2(M,E)$ is compact (even of finite rank) and, by \cref{eq:compact_subsets_in_L2_spaces_compact_operator}, there exists a compact subset $K \subseteq \interior M$ such that
 \begin{equation}
  \label{eq:below_essspec_function_choice_of_K_1}
  \int_{M\setminus K} \Big(\lambda + \delta + \frac{\varepsilon}{2}\Big) |s|^2 \,d\mu_g \leq \frac{\varepsilon}{2} \|s\|^2 + Q_A(s,s)
 \end{equation}
 for all $s \in \img{P_0} \subseteq \dom{Q_A}$ and
 \begin{equation}
  \label{eq:below_essspec_function_choice_of_K_2}
  \bigg(\int_{M\setminus K} |P_0s|^2\,d\mu_g\bigg)^{1/2} \leq \frac{\delta}{2(\lambda + \delta + \tfrac{\varepsilon}{2})} \|s\|
 \end{equation}
 for all $s \in L^2(M,E)$.
 Here, \cref{eq:below_essspec_function_choice_of_K_1} is possible since $s \mapsto (\tfrac{\varepsilon}{2} \|s\|^2 + Q_A(s,s))^{1/2}$ is equivalent to the norm on $\dom{Q_A}$, and \cref{eq:below_essspec_function_choice_of_K_2} works because $P_0 \colon H \to H$ is a finite rank projection.
 Now, for $s \in \dom{Q_A}$,
 \begin{multline}
  \label{eq:below_essspec_function_1}
  Q_A(s,s)
  = Q_A(P_0 s, P_0 s) + Q_A((1-P_0)s,(1-P_0)s) \geq \\
  \geq \int_{M\setminus K} \Big(\lambda + \delta + \frac{\varepsilon}{2}\Big) |P_0s|^2\,d\mu_g - \frac{\varepsilon}{2} \int_M |P_0s|^2\,d\mu_g + (\lambda + \delta) \int_M |(1-P_0)s|^2\,d\mu_g \geq \\
  \geq \llangle \chi P_0 s, P_0 s \rrangle + \llangle \chi (1-P_0)s,(1-P_0)s \rrangle
 \end{multline}
 with $\chi \coloneqq (\lambda + \delta + \frac{\varepsilon}{2})\chi_{M\setminus K} - \frac{\varepsilon}{2} = (\lambda + \delta)\chi_{M\setminus K} - \frac{\varepsilon}{2} \chi_K$, and where we have used that $\chi \leq \lambda + \delta$ to estimate the term with $(1-P_0)s$.
 The right hand side of \cref{eq:below_essspec_function_1} is equal to
 \begin{multline*}
 \llangle \chi s,s \rrangle - \llangle (P_0\chi(1-P_0)+(1-P_0)\chi P_0)s,s \rrangle = \\
 = \llangle \chi s, s \rrangle - \llangle P_0 (\chi + \tfrac{\varepsilon}{2}) (1-P_0)s, s \rrangle - \llangle (1-P_0)(\chi + \tfrac{\varepsilon}{2})P_0 s, s \rrangle,
 \end{multline*}
 where $P_0(1-P_0) = 0$ was used in order to replace $\chi$ by $\widetilde \chi \coloneqq \chi+\frac{\varepsilon}{2}$.
 Moreover,
 \begin{multline*}
  - \llangle P_0 \widetilde \chi (1-P_0)s, s \rrangle - \llangle (1-P_0) \widetilde \chi P_0 s, s \rrangle = \\
  = - \llangle \widetilde \chi s, P_0 s \rrangle
  + \llangle \widetilde \chi P_0 s, P_0 s \rrangle
  - \llangle P_0 s, \widetilde \chi s \rrangle
  + \llangle \widetilde \chi P_0 s, P_0 s \rrangle
  \geq -2\re \llangle \widetilde \chi s, P_0 s \rrangle,
 \end{multline*}
 the inequality being due to $\widetilde \chi \geq 0$.
 Now
 \begin{equation*}
  |2\llangle \widetilde \chi s,P_0 s \rrangle|
  = 2(\lambda+\delta + \tfrac{\varepsilon}{2}) |\llangle \chi_{M\setminus K} s,P_0 s \rrangle|
  \leq 2(\lambda+\delta + \tfrac{\varepsilon}{2}) \|s\| \|\chi_{M\setminus K} P_0 s\|
  \leq \delta \|s\|^2
 \end{equation*}
 by \cref{eq:below_essspec_function_choice_of_K_2}.
 Putting it all together, we have shown that
 \begin{equation}
  Q_A(s,s) \geq \llangle \chi s,s \rrangle - \delta \|s\|^2
  = \int_M \Big(\lambda \chi_{M\setminus K} - \Big(\frac{\varepsilon}{2} + \delta\Big) \chi_K\Big) |s|^2\,d\mu_g
  \geq \int_M (\lambda \chi_{M\setminus K} - \varepsilon \chi_K)|s|^2\,d\mu_g
 \end{equation}
 for all $s \in \dom{Q_A}$, as claimed.
\end{proof}

The next result is the appropriate formulation of \defemph{Persson's theorem} \cite{Persson1960} to our setting, and gives a characterization of the bottom of the essential spectrum, and its proof is now an easy consequence of \cref{decomposition_principle,below_essspec_function}.
A version for second order operators on $\R^n$ can also be found in \cite[Theorem~3.2]{Agmon1982}.

\begin{thm}
 \label{limit_of_infimum_of_spectrum}
 Let $A$ be a nonnegative self-adjoint extension of an elliptic differential operator acting on the sections of a Hermitian vector bundle $E \to M$ over a Riemannian manifold.
 Then
 \begin{equation}
  \label{eq:limit_of_infimum_of_spectrum}
  \lim_K \big(\inf\spec{A_{M\setminus K}}\big) = \inf\essspec{A},
 \end{equation}
 where the limit is with respect to the net of compact subsets of $\interior M$. \end{thm}

\begin{proof}
 Given $0 < \lambda < \inf\essspec{A}$, there exists a compact subset $K \subseteq \interior M$ such that $Q_A(s,s) \geq \lambda \int_{M\setminus K}|s|^2\,d\mu_g$ for all $s \in \dom{Q_A}$ with $s|_K = 0$, see \cref{below_essspec_function}.
 For $s \in \dom{Q_{A,M\setminus K}}$, we have $s_0 \in \dom{Q_A}$ by \cref{extension_by_zero_belongs_to_domain}, and
 \[ Q_{A,M\setminus K}(s,s) = Q_A(s_0,s_0) \geq \lambda \int_{M\setminus K} |s_0|^2\,d\mu_g = \lambda \|s\|^2. \]
 Therefore, $\inf\spec{A_{M\setminus K}} \geq \lambda$.
 It follows from \cref{extension_by_zero_belongs_to_domain} that $K \mapsto \inf\spec{A_{M\setminus K}}$ is an increasing net, so the limit \cref{eq:limit_of_infimum_of_spectrum} exists.
 Since the above holds for every $\lambda < \inf\essspec{A}$, we obtain $\lim_K(\inf\spec{A_{M\setminus K}}) \geq \inf\essspec{A}$, and by \cref{decomposition_principle} we also have
 \begin{equation*}
  \inf\spec{A_{M\setminus K}} \leq \inf\essspec{A_{M\setminus K}} = \inf\essspec{A},
 \end{equation*}
 for all compact $K \subseteq \interior M$,
 so that equality holds in \cref{eq:limit_of_infimum_of_spectrum}.
\end{proof}

In case $\spec{A}$ is discrete, we can use \cref{below_essspec_function} to construct proper coercivity functions for $Q_A$, in the following sense:

\begin{thm}
 \label{characterization_of_discreteness_of_spectrum}
 Let $A$ be a nonnegative self-adjoint extension of an elliptic differential operator acting on the sections of $E \to M$.
 Then the following are equivalent:
 \begin{enumerate}
 \item
  \label{item:characterization_of_discreteness_of_spectrum_disc_spec}
  The spectrum of $A$ is discrete, \ie $\essspec{A} = \emptyset$.
 \item
  \label{item:characterization_of_discreteness_of_spectrum_function_infty}
  There exists a proper  \footnote{    In this case, $\psi$ is proper if and only if $\psi(x) \to \infty$ as $x \to \infty$ (leaving every compact subset of $\interior M$).
  }
  smooth function $\psi \colon \interior M \to [-1,\infty)$ such that
  \begin{equation}
   \label{eq:characterization_of_discreteness_of_spectrum_function_infty}
   Q_A(s,s) \geq \int_M \psi|s|^2\,d\mu_g
  \end{equation}
  for all $s \in \dom{Q_A}$.
 \item
  \label{item:characterization_of_discreteness_of_spectrum_function_infty_measurable}
  There exists a proper measurable function $\psi \colon \interior M \to [-1,\infty)$ such that \cref{eq:characterization_of_discreteness_of_spectrum_function_infty} holds for all $s \in \dom{Q_A}$.
 \end{enumerate}
\end{thm}

\begin{proof}
 Item~\cref{item:characterization_of_discreteness_of_spectrum_function_infty} is inspired by \cite{Iwatsuka1986,Kondratiev2002,Haslinger2014,Ruppenthal2011a}, where the construction is done for certain classes of operators.
 Assume first that $A$ has discrete spectrum.
 By \cref{below_essspec_function}, there are compact subsets $K_k \subseteq \interior M$, $k \in \N$, such that $Q_A(s,s) \geq 2^k k\int_{M\setminus K_k} |s|^2\,d\mu_g - \int_{K_k} |s|^2\,d\mu_g$ for all $s \in \dom{Q_A}$.
 Without loss of generality, we may assume that $(K_k)_{k \in \N}$ forms a compact exhaustion of $\interior M$.
 For $s \in \dom{Q_A}$, we estimate
 \begin{multline*}
   Q_A(s,s) = \sum_{k=1}^\infty 2^{-k}Q_A(s,s)
   \geq \sum_{k=1}^\infty\bigg(\int_{M\setminus K_k} k|s|^2\,d\mu_g - 2^{-k}\int_{K_k} |s|^2\,d\mu_g\bigg) \geq \\
   \geq \sum_{k=1}^\infty \int_{K_{k+1}\setminus K_k} k|s|^2\,d\mu_g - \sum_{k=1}^\infty 2^{-k} \|s\|^2
   = \sum_{k=1}^\infty \int_{K_{k+1}\setminus K_k} k|s|^2\,d\mu_g - \|s\|^2.
 \end{multline*}
 Let $\psi_0 \colon \interior M \to [0,\infty)$ be a smooth function with $k-1 \leq \psi_0|_{(K_{k+1}\setminus K_k)} \leq k$ for $k \geq 1$ and $\psi_0|_{K_1} = 0$.
 Then $\psi_0$ is proper, and $\psi \coloneqq \psi_0 - 1 \colon \interior M \to [-1,\infty)$ has the properties sought in items~\cref{item:characterization_of_discreteness_of_spectrum_function_infty,item:characterization_of_discreteness_of_spectrum_function_infty_measurable}.

 Clearly, \cref{item:characterization_of_discreteness_of_spectrum_function_infty} implies \cref{item:characterization_of_discreteness_of_spectrum_function_infty_measurable}, and if $\psi \colon \interior M \to [-1,\infty)$ is as in \cref{item:characterization_of_discreteness_of_spectrum_function_infty_measurable}, then for $\lambda > 0$ fixed we put $K \coloneqq \psi^{-1}([-1,\lambda])$.
 Since $\psi$ is proper, $K$ is compact, and for $s \in \dom{Q_A}$ we have
 \[ Q_A(s,s) + \|s\|^2
 \geq \int_M (\psi + 1)|s|^2\,d\mu_g
 \geq \int_{M\setminus K} (\psi + 1)|s|^2\,d\mu_g
 \geq (\lambda+1)\int_{M\setminus K}|s|^2\,d\mu_g, \]
 hence $Q_A(s,s) \geq \lambda\int_{M\setminus K}|s|^2\,d\mu_g - \int_K |s|^2\,d\mu_g$.
 It follows that $\lambda \leq \inf\spec{A_{M\setminus K}}$, therefore $\lambda \leq \inf\essspec{A}$ by \cref{limit_of_infimum_of_spectrum}.
 Since $\lambda$ was arbitrary, $\essspec{A} = \emptyset$.
\end{proof}

\begin{rem}
 \label{rem:characterization_of_discreteness_of_spectrum}
 By modifying the definition of $K_k$ in the proof of \cref{characterization_of_discreteness_of_spectrum} such that $Q_A(s,s) \geq 2^kk\int_{M\setminus K_k}|s|^2\,d\mu_g - \varepsilon\int_{K_k}|s|^2\,d\mu_g$ for $s \in \dom{Q_A}$ and $k \in \N$, we see that, for every $\varepsilon > 0$, there is $\psi \colon \interior M \to [-\varepsilon,\infty)$ smooth, proper, and satisfying \cref{eq:characterization_of_discreteness_of_spectrum_function_infty}.
 By replacing $A$ with $A - (\inf\spec{A})\id{L^2(M,E)}$, we can also choose $\psi$ to have values in $[\inf\spec{A}-\varepsilon,\infty)$.
\end{rem}

\begin{rem}
 \Cref{limit_of_infimum_of_spectrum} also allows a characterization of the condition $\inf\essspec{A} > 0$.
 By basic results of spectral theory, this is equivalent to $\img{A}$ being closed and $\ker(A)$ having finite dimension, \ie $A$ is an (unbounded) Fredholm operator.
 Using \cref{limit_of_infimum_of_spectrum}, it is straightforward to see that this is the case if and only if there is a compact subset $K \subseteq \interior M$ and $C>0$ such that
 \begin{equation}
  \label{eq:characterization_of_fredholm_estimate}
  \|s\|^2 \leq C\bigg(Q_A(s,s) + \int_K |s|^2\,d\mu_g\bigg)
 \end{equation}
 holds for all $s \in \dom{Q_A}$.
 In \cite[section~3.1]{Ma2007}, the inequality~\cref{eq:characterization_of_fredholm_estimate} is called a \defemph{fundamental estimate} and shown to be sufficient (with $A = \square^E$) for $A$ to be Fredholm.
 Also equivalent to \cref{eq:characterization_of_fredholm_estimate} is the existence of $K \subseteq \interior M$ compact and $C>0$ such that $\|s\|^2 \leq CQ_A(s,s)$ for all $s \in \dom{Q_A}$ with $\supp(s) \subseteq M\setminus K$.
\end{rem}

 \section{Preliminaries on complex and Hermitian geometry}
\label{sec:complex_geometry}

Let $X$ be a Hermitian manifold, with almost complex structure $J$ and compatible Riemannian metric $g$, and let $E \to X$ be a Hermitian holomorphic vector bundle.
On the complex vector bundle
$(TX\otimes_\R\Cplx,\Cplxi)$ we have the Hermitian metric $\langle \arghere,\arghere\rangle$, defined as the sesquilinear extension of $g$. Together with the Hermitian metric on $E$, this induces Hermitian forms on the bundles $\Lambda^k T^*X \otimes E$, which we all continue to denote by $\langle\arghere,\arghere\rangle$.
On functions, we put $\langle f, g \rangle \coloneqq f\ol g$, as usual.
These also induce a global inner product on $\Omega_c(X,E)$, the smooth differential forms on $X$ with values in $E$ and with compact support, given by~\cref{eq:ltwo_inner_product}, and requiring that $\llangle u,v \rrangle = 0$ if $u$ and $v$ have different degree. 
Since $X$ is Hermitian, it follows that the decomposition $\Omega_c(X,E) = \bigoplus_{p,q}\Omega_c^{p,q}(X,E)$ is orthogonal for this inner product.
We will frequently make use of local orthonormal frames.
Usually, $\{w_j\}_{j=1}^n$ will denote such a frame for $T^{1,0}X$, with its conjugate frame $\{\ol w_j\}_{j=1}^n$ a local orthonormal frame of $T^{0,1}X$.
We have the dual coframes $\{w^j\}_{j=1}^n$ and $\{\ol w^j\}_{j=1}^n$ of $(T^{1,0}X)^*$ and $(T^{0,1}X)^*$, respectively.

We denote by $\ins{\xi}$ for $\xi \in TX\otimes_\R\Cplx$ (or a vector field) the insertion operator on forms, and by $\extp{\alpha}$ the operator of taking the wedge product with $\alpha$.
Then $\ins{\xi}$ is the adjoint to $\extp{\xi^\flat}$, with $\xi^\flat$ the dual one-form.
By writing $u,v \in \Lambda^{p,q}T^*_xX \otimes E_x$ in terms of an orthonormal frame as above, it is easily seen that
\begin{equation}
 \label{eq:exterior_algebra_identity_1}
 \sum_{j=1}^n \langle \ins{w_j}(u),\ins{w_j}(v) \rangle  = p \langle u,v \rangle \quad\text{and}\quad \sum_{j=1}^n \langle\ins{\ol w_j}(u),\ins{\ol w_j}(v)\rangle = q\langle u,v \rangle.
\end{equation}
In other words, $\sum_{j=1}^n (\ins{w_j})^* \ins{w_j} = \sum_{j=1}^n \extp{w^j} \ins{w_j}$ is $p$ times the identity on $\Lambda^{p,\bullet}T^*X \otimes E$.
This also implies
\begin{equation}
 \label{eq:exterior_algebra_identity_2}
 \sum_{j=1}^n \big\langle (\alpha \wedge \ins{\xi})\ins{\ol w_j}(u), \ins{\ol w_j}(v) \big\rangle = (q-1) \big\langle (\alpha \wedge \ins{\xi})u,v \big\rangle 
\end{equation}
for all $\alpha \in T^*_x X \otimes_\R\Cplx$ and $\xi \in T_xX \otimes_\R\Cplx$.

Associated to the Dolbeault complex~\cref{eq:dolbeault_complex} is the second order differential operator\footnote{  We allow ourselves a small abuse of notation here, since the same symbol is used to denote the self-adjoint extension of $\square^E$ with $\dbar$-Neumann boundary conditions.
}
\[ \square^E \coloneqq \dbar^{E,\fadj}\dbar^E + \dbar^E\dbar^{E,\fadj} = (\dbar^E + \dbar^{E,\fadj})^2 \colon \Omega^{\bullet,\bullet}(X,E) \to \Omega^{\bullet,\bullet}(X,E), \]
called the \defemph{Dolbeault Laplacian} (or \defemph{complex Laplacian}), where we denote by
$
\dbar^{E,\fadj}
$
the formal adjoint to $\dbar^E$ with respect to \cref{eq:ltwo_inner_product}.
The principal symbol of $\square^E$ reads
\begin{equation}
 \label{eq:symbol_square}
 \symb{\square^E}(\gamma)u = - \ins{(\gamma^\sharp)^{0,1}}(\gamma^{0,1} \wedge u) - \gamma^{0,1} \wedge \ins{(\gamma^\sharp)^{0,1}}(u) = -\langle \gamma^{0,1},\gamma^{0,1} \rangle\, u  = -\frac{1}{2} |\gamma|^2\, u,
\end{equation}
for all $\gamma \in T^*_xX \subseteq T^*_xX \otimes_\R \Cplx$ and $u \in \Lambda^{\bullet,\bullet} T^*_xX \otimes E_x$,
where $\ins{Z}$ for $Z \in TX \otimes_\R\Cplx$ is the insertion operator, and where $\gamma^\sharp$ is the dual vector, with $(0,1)$-part $(\gamma^\sharp)^{0,1}$.

\subsection{Weitzenb\"ock type formulas for the Dolbeault Laplacian}

It follows from \cref{eq:symbol_square} that $2\square^E$ is an operator of Laplace type, meaning that its principal symbol is $\symb{2\square^E}(\gamma) = - |\gamma|^2 \id{\Lambda T^*X \otimes E}$. Consequently, $\sqrt{2}(\dbar^E + \dbar^{E,\fadj})$ is a Dirac type operator.
On a \emph{K\"ahler} manifold, this is an important example of a Dirac operator associated to a Dirac bundle in the sense of \cite{Lawson1989}.
In fact,
\begin{equation}
 \label{eq:clifford_kaehler}
 c_p(\gamma)u \coloneqq \sqrt{2} \big(\gamma^{0,1} \wedge u - \ins{(\gamma^\sharp)^{0,1}}(u)\big)
\end{equation}
defines a Clifford module structure on $\Lambda^{p,\bullet}T^*X \otimes E$ such that $(\Lambda^{p,\bullet}T^*X \otimes E,X,c_p,\widetilde\nabla)$ is a Dirac bundle, where $\widetilde\nabla$ is the connection induced by the Levi--Civita connection on $TX$ and the Chern connection on $E$.
For a reference on this, see \cite[Proposition~3.27]{Berline2004}.
The Dirac operator associated to this structure is $D^E \coloneqq \sqrt{2} (\dbar^E + \dbar^{E,\fadj})$.
The above implies that, on a K\"ahler manifold, we have the Weitzenb\"ock type formula
\begin{equation}
 \label{eq:dolbeault_weitzenboeck}
 2\square^E = (D^E)^2 = \widetilde \nabla^\fadj \widetilde\nabla + c_p(R^{\Lambda^{p,\bullet}T^*X \otimes E})
\end{equation}
on $\Omega^{p,\bullet}(X,E)$, where the last term is the endomorphism of $\Lambda^{p,\bullet}T^*X \otimes E$ defined by
\begin{equation}
 \label{eq:general_weitzenboeck_curvature_term}
 c_p(R^{\Lambda^{p,\bullet}T^*X \otimes E})|_x
 = \sum_{j < k} c_p(e^j)\circ c_p(e^k)\circ R^{\Lambda^{p,\bullet}T^*X\otimes E}(e_j,e_k)
\end{equation}
for any (real) orthonormal basis $\{e_j\}_{j=1}^{2n}$ of $T_xX$, see \cite[Theorem~II.8.2]{Lawson1989} or \cite[Theorem~11.1.67]{Nicolaescu2014}.
We shall not require the precise form of the curvature term in \cref{eq:dolbeault_weitzenboeck}, but see \cite{Ma2007} for an explicit computation in the case $p=0$.

Another Weitzenb\"ock type formula for a Hermitian holomorphic vector bundle $E$ over a K\"ahler manifold $(X,\omega)$ is the \defemph{Bochner--Kodaira--Nakano} formula
\begin{equation}\label{eq:bkn}
 \square^E =  \big(d^E_{1,0}d^{E,\fadj}_{1,0} + d^{E,\fadj}_{1,0}d^E_{1,0}\big) + [\Cplxi R^E \evwedge,\Lambda]
\end{equation}
which also expresses $\square^E$ as the sum of a Laplace type operator and a zeroth order term.
In \cref{eq:bkn}, $d^E$ denotes the exterior covariant derivative (see \cref{eq:exterior_covariant_derivative}) associated to the Chern connection on $E$, with $(1,0)$-part $d^E_{1,0}$ (and $(0,1)$-part $\dbar^E$) and $d^{E,\fadj}_{1,0}$ the formal adjoint.
Furthermore, $\Lambda \colon \Lambda^{\bullet,\bullet}T^*X \otimes E \to \Lambda^{\bullet-1,\bullet-1}T^*X \otimes E$ is the adjoint to exterior multiplication with the K\"ahler form $\omega$, the wedge product $\evwedge$ is combined with the evaluation map $\en{E}\otimes E \to E$, and $[\arghere,\arghere]$ is the commutator of endomorphisms.
For a proof, see \cite[Theorem~2.7]{Ohsawa2015}.
Formula~\cref{eq:bkn} has an extension to Hermitian manifolds that are not K\"ahler, with additional torsion terms occurring.
This is due to Demailly~\cite{Demailly1986}, and a proof can also be found in \cite[Theorem~1.4.12]{Ma2007}.

As in the introduction, let $M \subseteq X$ be the closure of a smoothly bounded open subset.

\begin{defn}
 \label{def:B_M_spaces}
 Suppose that $U \subseteq M$ is a (relatively) open subset.
 We define
 \begin{equation}
  \label{eq:B_M_spaces}
  B_{\partial M}(U,E) \coloneqq
  \big\{u \in \Omega_c(U,E) :
  \symb{\dbar^{E,\fadj}}(\nu)u|_{\partial M \cap U} =
  -\ins{(\nu^{0,1})}(u)|_{\partial M \cap U} = 0
  \big\},
 \end{equation}
 where $\nu$ is a normal vector field to $\partial M$, and $\nu^{0,1} = \frac{1}{2}(\nu + \Cplxi J\nu)$ is its component in $T^{0,1}M$.  We denote by $B_{\partial M}^{p,q}(U,E)$ the forms of bidegree $(p,q)$ in $B_{\partial M}(U,E)$.
\end{defn}

Note that if $u \in B_{\partial M}(U,E)$, then also $\ins{\xi}(u) \in B_{\partial M}(U,E)$ for every vector field $\xi$ on $U$, since insertion operators anticommute. Using integration by parts, one can derive the following global version of \cref{eq:bkn}.
We refer to \cite[Theorem~1.4.21]{Ma2007} for a proof.

\begin{thm}[Global Bochner--Kodaira--Nakano formula]
 \label{BKN_boundary}
 Let $M \subseteq X$ be as above, with $X$ K\"ahler, and suppose $E \to M$ is a Hermitian holomorphic vector bundle.
 Then
 \begin{equation}
  \label{eq:BKN_boundary}
  \|\dbar^E u\|^2 + \|\dbar^{E,\fadj}u\|^2 = \|d^{E,\fadj}_{1,0}u\|^2 + \llangle \Cplxi R^E \evwedge \Lambda u,u \rrangle + \int_{\partial M} \mathscr L(u,u)\,d\mu_{\partial M}
 \end{equation}
 holds for all $u \in B_{\partial M}^{n,\bullet}(M,E)$.
\end{thm}

The boundary integral in \cref{eq:BKN_boundary} contains the Levi form of $\partial M$.
If $\rho \in C^\infty(X,\R)$ is a defining function for $M$, meaning $M =\rho^{-1}((-\infty,0))$ and $|d\rho|=1$ on $\partial M$, then
\begin{equation}
 \label{eq:modified_levi_form_defining_function}
 \mathscr L(u,v) = \sum_{j,k=1}^n \partial\dbar\rho(w_j,\ol w_k) \langle \ol w^k \wedge \ins{\ol w_j}(u),v \rangle
\end{equation}
on $\partial M$, independent of the chosen defining function.
The manifold with boundary $M \subseteq X$ is called \defemph{$q$-Levi pseudoconvex} if $\mathscr L(\alpha,\alpha) \geq 0$ holds for all $\alpha \in B_{\partial M}^{0,q}(M,\Cplx)$.
The case $q=1$ corresponds to the usual notion of Levi pseudoconvexity.
Note that, according to this definition, the closure $M$ of every smoothly bounded open subset of $X$ is $n$-Levi pseudoconvex, with $n$ the complex dimension of $X$.
Indeed, this follows from the boundary conditions that elements of $B_{\partial M}^{0,n}(M,\Cplx)$ must satisfy, and using an orthonormal basis in \cref{eq:modified_levi_form_defining_function} that includes $\sqrt{2} \nu^{1,0}$.
If $\mathscr L(\alpha,\alpha) \geq 0$ for $\alpha \in B_{\partial M}^{0,1}(M,\Cplx)$, then this inequality continues to hold for $\alpha \in B_{\partial M}^{p,q}(M,\Cplx)$, with $q \geq 1$, see the arguments in \cref{eq:nakano_lower_semiboundedness_percolates} below.
Following the same reasoning, if $M$ is $q$-Levi pseudoconvex, then it is also $q'$-Levi pseudoconvex for every $q' \geq q$.

For $X=\Cplx^n$, equation~\cref{eq:BKN_boundary} is also referred to as the \defemph{Morrey--Kohn--H\"ormander formula}, see \cite[Proposition~2.4]{Straube2010} or \cite[Proposition~4.3.1]{Chen2001}.
Original works include \cite{Hoermander1965,Morrey1958}, but see \cite{Straube2010} for extensive references.

\begin{rem}
 \label{rem:nakano_lower_semibounded}
 \begin{inlineenum}
 \item
  Using $\Lambda = \Cplxi \sum_{j=1}^n \ins{w_j}\ins{\ol w_j}$, it is not hard to see that, on $\Lambda^{n,\bullet} T^*M \otimes E$, the operator $\Cplxi R^E \evwedge \Lambda$ which appears in \cref{eq:BKN_boundary} has the form
  \begin{equation}
   \label{eq:curvature_wedge_lambda}
   \Cplxi R^E \evwedge \Lambda u = \sum_{j,k=1}^n R^E(w_j,\ol w_k) \extp{\ol w^k}\ins{\ol w_j}(u).
  \end{equation}
  Hence, the condition of $q$-Nakano lower semiboundedness from~\cref{eq:nakano_lower_semibounded} can equivalently be described as
  $$ \langle \Cplxi R^E \evwedge \Lambda u, u\rangle \geq c|u|^2 $$
  for all $u \in \Lambda^{n,q}T^*M \otimes E$, see \cref{eq:curvature_wedge_lambda}.
  If \cref{eq:nakano_lower_semibounded} continues to hold for $u \in \Lambda^{p,q}T^*_xM \otimes E_x$ with $0 \leq p \leq n$.
  Moreover, if $E$ is $(q-1)$-Nakano lower semibounded, then it is also $q$-Nakano lower semibounded.
  This can be seen by induction:
  if \cref{eq:nakano_lower_semibounded} is true on $\Lambda^{p,q-1}T^*_x M \otimes E_x$ with $q \geq 2$, then
  \begin{multline}
   \label{eq:nakano_lower_semiboundedness_percolates}
   \sum_{j,k=1}^n \langle (\extp{\ol w^k} \circ \ins{\ol w_j}) \otimes R^E(w_j, \ol w_k)u,u \rangle = \\
   = \frac{1}{q-1} \sum_{m=1}^n \sum_{j,k=1}^n \langle ((\extp{\ol w^k} \circ \ins{\ol w_j}) \otimes R^E(w_j,\ol w_k)) \ins{\ol w_m}(u),\ins{\ol w_m}(u) \rangle \geq \frac{n}{q-1}\, c
  \end{multline}
  for all $u \in \Lambda^{p,q}T^*_x M \otimes E_x$ by \cref{eq:exterior_algebra_identity_2}.
  In particular, $\nakanoc[q]{E} \geq \frac{n}{q-1} \nakanoc[q-1]{E}$, with $\nakanoc[q]{E}$ the largest possible constant $c$ in~\cref{eq:nakano_lower_semibounded}.

 \item
  Due to H\"older's inequality, we always have
  \begin{equation*}
   \sum_{j,k=1}^n \langle (\extp{\ol w^k} \circ \ins{\ol w_j}) \otimes R^E(w_j,\ol w_k) u,u \rangle    \geq - \sum_{j,k=1}^n |R^E(w_j,\ol w_k)| |u|^2    \geq - n |R^E| |u|^2.
  \end{equation*}
  Thus, if $R^E$ is bounded
  , then $E$ is Nakano lower semibounded.
  In particular, if $(M,J,g)$ is a K\"ahler manifold of $0$-bounded geometry, see \cref{sec:bounded_geometry}, then $TM$ (hence also $T^{1,0}M$) is a Nakano lower semibounded vector bundle.

 \item
  As an example, if $E \to M$ and $F \to M$ are two $q$-Nakano lower semibounded vector bundles, then the Whitney sum $E \oplus F$ and the tensor product $E \otimes F$ are again $q$-Nakano lower semibounded, with $\nakanoc[q]{E \oplus F} = \min\{\nakanoc[q]{E},\nakanoc[q]{F}\}$ and $\nakanoc[q]{E \otimes F} = \nakanoc[q]{E} + \nakanoc[q]{F}$.
 \end{inlineenum}
\end{rem}

Formula~\cref{eq:BKN_boundary} has an extension to $(p,q)$-forms for $0 \leq p \leq n$, with a term involving the curvature of $T^{1,0}M$ occurring.
Consider the morphism of complex vector bundles
\begin{equation}
 \label{eq:definition_of_Psi}
 \Psi^E_p \colon \Lambda^{p,\bullet}T^*M \otimes E \to \Lambda^{n,\bullet}T^*M \otimes (E\otimes \Lambda^{n-p,0}TM),
 \quad
 \Psi^E_p(u) = \psum_{|J|=n-p} (w^J \wedge u) \otimes w_J,
\end{equation}
where as usual the primed sum means that the summation is done over all increasing maps $J \colon \{1,\dotsc,n-p\} \to \{1,\dotsc,n\}$, \ie all subsets of $\{1,\dotsc,n\}$ of cardinality $n-p$, and $w^J \coloneqq w^{J(1)} \wedge \dotsb \wedge w^{J(n-p)}$, with analogous definition for $w_J$.
From this, it is immediate that $\Psi^E_p$ is an isometry, with inverse given by the contraction map $\Lambda^{n,0}T^*M \otimes \Lambda^{n-p,0}TM \to \Lambda^{p,0}T^*M$, up to a sign factor.
Using local holomorphic sections of $\Lambda^{n-p,0}TM$, $\Lambda^{n,0}T^*M$, and $E$, one readily shows that $\dbar^{E\otimes \Lambda^{n-p,0}TM} \circ \Psi^E_p = (-1)^{n-p} \Psi^E_p \circ \dbar^E$, and because $\Psi^E_p$ is an isometry, it also intertwines the formal adjoints of the relevant Dolbeault operators.
It is also clear that $\Psi^E_p$ maps $B_{\partial M}^{p,\bullet}(U,E)$ to $B_{\partial M}^{n,\bullet}(U,E \otimes \Lambda^{n-p,0}TM)$ for every open $U \subseteq M$.

\begin{cor}
 \label{bkn_boundary_localized}
 For any open subset $U \subseteq M$, and any $u \in B_{\partial M}^{p,\bullet}(U,E)$, we have
 \begin{equation}
  \label{eq:bkn_boundary_localized}
  \|\dbar^E u\|^2 + \|\dbar^{E,\fadj} u\|^2
  = \big\|d^{E \otimes \Lambda^{n-p,0}TM,\fadj}_{1,0} \widetilde u\big\|^2
  + \big\llangle \Cplxi R^{E\otimes \Lambda^{n-p,0}TM} \evwedge \Lambda \widetilde u,\widetilde u \big\rrangle
  + \int_{\partial M \cap U} \mathscr L(u,u)\,d\mu_{\partial M}
 \end{equation}
 with $\widetilde u \coloneqq \Psi^E_p(u) \in B_{\partial M}^{n,\bullet}(U,E\otimes \Lambda^{n-p,0}TM)$.
\end{cor}

\begin{proof}
 This is immediate from \cref{BKN_boundary}, the above discussion on $\Psi^E_p$, using that $u$ is supported on $U$, and observing that $\mathscr L(\Psi^E_p(u),\Psi^E_p(u)) = \mathscr L(u,u)$.
\end{proof}

\subsection{More on the \texorpdfstring{$\dbar^E$}{dbar-E}-Neumann problem}
\label{sec:more_on_dbar}

We denote by $Q^E$ the quadratic form associated to the self-adjoint operator $\square^E$, as defined through spectral theory (see \cite{Schmuedgen2012} for an introduction to quadratic forms on Hilbert spaces).
Then $Q^E$ contains the same information as $\square^E$, and since $\square^E$ is nonnegative, it holds that $\dom{Q^E} = \dom{(\square^E)^{1/2}}$ and $Q^E(u,v) = \llangle (\square^E)^{1/2}u, (\square^E)^{1/2}v \rrangle$.
As $\square^E$ is the Laplacian of a Hilbert complex,\footnote{  By this we mean a (co)chain complex of closed and densely defined operators between Hilbert spaces, see \cite{Bruening1992} or also \cite{Berger2016}.
}
its quadratic form has the more accessible expression
$$ Q^E(u,v) = \llangle \dbar^E_\maxx u, \dbar^E_\maxx v \rrangle + \llangle \dbar^{E,*}_\maxx u, \dbar^{E,*}_\maxx v \rrangle $$
for $u,v \in \dom{Q^E} = \dom{\dbar^E_\maxx} \cap \dom{\dbar^{E,*}_\maxx}$, see the arguments in \cite[Proposition~2.3]{Goldshtein2011}.
One can show (see, \eg \cite{Folland1972}) that
\begin{enumerate}
\item\label{item:dbar_smooth_domain_of_adjoint}
 $\Omega_c(M,E) \cap \dom{\dbar^{E,*}_\maxx} = B_{\partial M}(M,E)$,
\item\label{item:dbar_smooth_domain_of_quadratic_form}
 $\Omega_c(M,E) \cap \dom{Q^E} = B_{\partial M}(M,E)$,
 and
\item\label{item:dbar_smooth_domain_of_laplacian}
 $\Omega_c(M,E) \cap \dom{\square^E} = \{u \in B_{\partial M}(M,E) : \dbar^E u \in B_{\partial M}(M,E)\}$.
\end{enumerate}
In particular, $\dbar^{E,*}_\maxx u = \dbar^{E,\fadj}u$ for $u \in B_{\partial M}(M,E)$ by \cref{eq:extension_of_diff_op_restrictions_to_diff_op}.

If $U \subseteq M$ is (relatively) open, then we denote by $\square^E_U$ the self-adjoint operator $(\square^E)_U$ on $L^2_{\bullet,\bullet}(U,E) = L^2_{\bullet,\bullet}(U,E)$, see \cref{def:localized_differential_operator}, with associated quadratic form $Q^E_U \coloneqq Q_{\square^E,U}$.
We write $\square^E_{p,q}$ and $\square^E_{U,p,q}$ for the restrictions of $\square^E$ and $\square^E_U$ to $L^2_{p,q}(M,E)$ and $L^2_{p,q}(U,E)$, respectively.

\begin{rem}
 The quadratic form $Q^{E|_{U}}$ is an extension of $Q^E_U$, in the sense that $\dom{Q^E_U} \subseteq \dom{Q^{E|_{U}}}$ and
 \begin{equation}
  \label{eq:dbar_localization_and_restriction}
  Q^{E|_{U}}(u,u) = Q^E_U(u,u)
 \end{equation}
 for all $u \in \dom{Q^E_U}$.
 Intuitively, this is because $Q^E_U$ requires Dirichlet boundary conditions on $\partial U \cap \interior M$, while the self-adjoint operator associated to $Q^{E|_{U}}$ only requires the weaker $\dbar$-Neumann boundary conditions.
 
 To formally show \cref{eq:dbar_localization_and_restriction}, let $u \in \dom{Q^E_U}$.
 Then $u_0$, defined as the extension of $u$ to $M$ by zero, see \cref{localized_form_is_closable}, belongs to $\dom{Q^E} =\dom{\dbar^E_\maxx} \cap \dom{\dbar^{E,\fadj}_\minn}$, and clearly $u = (u_0)|_{U} \in \dom{\dbar^{E|_{U}}_\maxx}$.
 For all $k \in \N$, we find $v_k \in \Omega_\cc(U,E)$ such that $v_k \to u$ in $L^2_{\bullet,\bullet}(U,E)$ and $\|\dbar^{E,\fadj}_\minn u_0 - \dbar^{E,\fadj} v_k\| \leq \frac{1}{k}$.
 This may be done by first approximating $u_0$ by $\widetilde v_k \in \dom{Q^E} \subseteq \dom{\dbar^{E,\fadj}_\minn}$ with $\supp(\widetilde v_k) \subseteq U$, and then approximating, in the norm of $\dom{\dbar^{E,\fadj}_\minn}$, each $\widetilde v_k$ by elements of $\Omega_\cc(U,E)$.
 Since $(\dbar^{E,\fadj} v_k)|_{U} = \dbar^{E|_{U},\fadj} (v_k|_{U})$, it follows that $(v_k|_{U})_{k \in \N}$ is Cauchy in $\dom{\dbar^{E|_{U},\fadj}_\minn}$, hence converges to $u$ in this space due to the convergence in $L^2_{\bullet,\bullet}(U,E)$.
 Thus, $u$ belongs to $\dom{\dbar^{E|_{U},\fadj}_\minn}$, and
 \[ \dbar^{E|_{U},\fadj}_\minn u = \lim_k (\dbar^{E,\fadj}v_k)|_{U} = (\dbar^{E,\fadj}_\minn u_0)|_{U}, \]
 so that $Q^{E|_{U}}(u,u) = Q^E(u_0,u_0) = Q^E_U(u,u)$, as claimed.
\end{rem}

\begin{prop}
 \label{form_core_for_square}
 If $U \subseteq M$ is open, then the space $B_{\partial M}(U,E)$ from \cref{def:B_M_spaces} is a form core for $\square^E_U$.
\end{prop}

\begin{proof}
 We shall use the known fact that $B_{\partial M}(M,E)$ is a form core for $\square^E$ if $M$ is compact, \footnote{   The statement can be found in \cite[Lemma~3.5.1]{Ma2007}, where a reference is made to \cite[Proposition~1.2.4]{Hoermander1965}.
   A proof for $M$ a domain in $\Cplx^n$ can also be found in \cite[Proposition~2.3]{Straube2010}.
 }
 the proof of which requires careful use of mollifiers.
 We first treat the case $U=M$.
 By \cref{ex:dirac_operator_squared_leibniz_rule,ex:hilbert_complexes_leibniz_rule,bounded_symbol_compact_supported_sections_dense}, we know that the elements of $\dom{Q^E} = \dom{\dbar^E_\maxx+\dbar^{E,*}_\maxx}$ with compact support in $M$ are dense in $\dom{Q^E}$.
 If $u \in \dom{Q^E}$ has compact support, choose a compact manifold with boundary $X \subseteq M$ such that $\supp(u) \subseteq V \coloneqq (\partial M \cap X) \cup \interior X$, an open subset of $M$.
 Then $u|_V \in \dom{Q^E_V} \subseteq \dom{Q^{E|_V}} = \dom{Q^{E|_X}}$, see \cref{eq:dbar_localization_and_restriction}, and by the aforementioned result for compact manifolds, there exist $v_k \in B_{\partial X}(X,E)$ with $v_k \to u|_{X}$ as $k \to \infty$ in $\dom{Q^{E|_X}}$.
 Let $\varphi \in C^\infty_c(M,[0,1])$ with $\varphi = 1$ on $\supp(u)$.
 Then $\varphi v_k \in \Omega_c(M,E)$ and $\ins{(\nu^{0,1})}(\varphi v_k) = 0$ on $\partial M \cap \partial X$, and $\varphi v_k = 0$ on $\partial M \setminus \partial X$ anyways, so $\varphi v_k \in B_{\partial M}(M,E)$.
 By \cref{eq:extension_of_diff_op_restrictions_to_diff_op} and since $\sqrt{2}(\dbar^E + \dbar^{E,\fadj})$ is a Dirac type operator, we have the estimate
 \begin{equation}
  \label{eq:form_core_for_square_proof_1}
  Q^E(\varphi(v_k - v_j),\varphi(v_k - v_j))
  \leq \|d\varphi\|^2_{L^\infty(M,T^*M)} \|v_k - v_j\|_{L^2_{\bullet,\bullet}(\interior X,E)}^2 + 2 Q^{E|_X}(v_k - v_j,v_k - v_j).
 \end{equation}
 Thus, $(\varphi v_k)_{k \in \N}$ is Cauchy in $\dom{Q^E}$, hence convergent, and the limit agrees with $u$ by the convergence in $L^2_{\bullet,\bullet}(M,E)$.

 Now let $U \subseteq M$ be an arbitrary open subset.
 By the definition of $Q^E_U$, it suffices to show that every $u|_U$ with $u \in \dom{Q^E}$ and $\supp(u) \subseteq U$ can be approximated in the norm of $\dom{Q^E_U}$ by elements of $B_{\partial M}(U,E)$.
 By the above, we obtain $u_k \in B_{\partial M}(M,E)$ with $u_k \to u$ in $\dom{Q^E}$.
 Let $\varphi \in C^\infty(M,[0,1])$ be such that $\supp(\varphi) \subseteq U$ and $\varphi|_{\supp(u)} = 1$.
 Clearly, $\varphi u_k|_U \in B_{\partial M}(U,E)$, and a computation as in \cref{eq:form_core_for_square_proof_1} again gives convergence of $\varphi u_k|_U$ to $u|_U$ in $\dom{Q^E_U}$.
\end{proof}

Note that if $M$ is complete and without boundary, then $\square^E$ is essentially self-adjoint on $\Omega_c^{\bullet,\bullet}(M,E) = \Omega_\cc^{\bullet,\bullet}(M,E)$, as is indeed the case for all positive integer powers of formally self-adjoint first order differential operators on complete manifolds that satisfy the symbol bound~\cref{eq:bounded_symbol_compact_supported_sections_dense}, see \cite[Theorem~2.2]{Chernoff1973}.

 \section{Proof of \cref{spectral_percolation_specific}}
\label{sec:proof_of_percolation}

We establish \cref{spectral_percolation_specific} through a series of auxiliary results.

\begin{lem}
 \label{percolation_derivative_estimate}
 Let $U \subseteq M$ be open and suppose that $(w_j)_{j=1}^n$ is an orthonormal frame of $T^{1,0}U$.
 Then
 \[ \sum_{j=1}^n \big|d^{E,\fadj}_{1,0}(\ins{\ol w_j}(u))\big|^2 \leq 2q\big|d^{E,\fadj}_{1,0} u\big|^2 + \Big(2n \max\big\{|\nabla \ol w_k|^2 : 1 \leq k \leq n\big\}\Big) |u|^2 \]
 pointwise on $U$ for every $u \in \Omega^{p,q}(U,E)$.
\end{lem}

\begin{proof} Let $X$ be a complex vector field on $M$.
 We have
 \[
 \ins{X} \circ d^{E,\fadj}_{1,0}
 = (d^E_{1,0} \circ \extp{X^\flat})^\fadj
 = \extp{\partial(X^\flat)}^\fadj - (\extp{X^\flat} \circ d^E_{1,0})^\fadj
 = \extp{\partial(X^\flat)}^\fadj - d^{E,\fadj}_{1,0}\circ \ins{X},
 \]
 where $\extp{\alpha}$ is exterior multiplication with $\alpha \in \Lambda^{\bullet,\bullet}T^*M$.
 Therefore,
 \begin{align*}
  \sum_{j=1}^n \big|d^{E,\fadj}_{1,0}(\ins{\ol w_j}(u))\big|^2
   & \leq 2 \sum_{j=1}^n \big|\ins{\ol w_j}\big(d^{E,\fadj}_{1,0} (u)\big)\big|^2 + 2 \sum_{j=1}^n |\extp{\partial\ol w^j}^\fadj u|^2 \\
   & \leq 2q|d^{E,\fadj}_{1,0}u|^2 + 2 \sum_{j=1}^n |\extp{\partial\ol w^j}^\fadj|^2 |u|^2
 \end{align*}
 by \cref{eq:exterior_algebra_identity_1}, where $|\extp{\partial\ol w^j}^\fadj|$ denotes the (fiberwise) operator norm.
 Now
 $
 |\extp{\partial \ol w^j}^\fadj|
 \leq |\partial \ol w^j|
 \leq |\nabla \ol w^j|
 = |\nabla \ol w_j|$ ,
 which finishes the proof.
\end{proof}

\begin{lem}
 \label{spectral_percolation_localized}
 Let $U \subseteq M$ be an open subset of $M$ with trivial tangent bundle.
 Let $0 \leq p \leq n$, $1 \leq q \leq n$, and assume that $M$ is $q$-Levi pseudoconvex at every point of $U \cap \partial M$, and that $E|_U \otimes \Lambda^{n-p,0}TU$ is $q$-Nakano lower semibounded.
 Then
 \begin{equation}
  \label{eq:spectral_percolation_localized}
  \frac{1}{q}\sum_{j=1}^n Q^E_U(\ins{\ol w_j}(u),\ins{\ol w_j}(u)) \leq 2 Q^E_U(u,u) + \frac{2n \kappa - c(q+1)}{q} \|u\|^2
 \end{equation}
 for all $u \in B_{\partial M}^{p,q}(U,E)$ and every orthonormal frame $(w_j)_{j=1}^n$ of $T^{1,0}X|_U$, and where $c \coloneqq \nakanoc[q]{E\otimes \Lambda^{n-p,0}TM}$ and $\kappa \coloneqq \max\big\{|\nabla \ol w_j|^2 : 1 \leq j \leq n\big\}$.
\end{lem}

\begin{proof}
 The orthonormal frame $(w_j)_{j=1}^n$ from our assumption induces an isometry $L^2_{p,q}(U,E) \to L^2_{p,q-1}(U,E)^{\oplus n}$, given by $u \mapsto \frac{1}{\sqrt{q}}(\ins{\ol w_j}(u))_{j=1}^n$, see \cref{eq:exterior_algebra_identity_1}.
 By the global Bochner--Kodaira--Nakano formula~\cref{eq:bkn_boundary_localized} we have, for every $u \in B_{\partial M}^{p,q}(U,E)$,
 \begin{multline*}   \frac{1}{q}\sum_{j=1}^n Q^E_U(\ins{\ol w_j}(u),\ins{\ol w_j}(u))
  = \frac{1}{q}\sum_{j=1}^n \bigg(
  \big\|d^{E\otimes \Lambda^{n-p,0}TM,\fadj}_{1,0}(\ins{\ol w_j}(\widetilde u))\big\|^2 + \\
  + \big\llangle \Cplxi R^{E\otimes \Lambda^{n-p,0}TM} \evwedge \Lambda \big(\ins{\ol w_j}(\widetilde u)\big),\ins{\ol w_j}(\widetilde u) \big\rrangle
  + \int_{\partial M \cap U} \mathscr L(\ins{\ol w_j}(u),\ins{\ol w_j}(u))\,d\mu_{\partial M}
  \bigg),
 \end{multline*}
 where $\widetilde u \coloneqq \Psi^E_p(u)$, which by using \cref{eq:exterior_algebra_identity_2} as well as \cref{percolation_derivative_estimate} (recall from~\cref{eq:curvature_wedge_lambda} the local formula for $\Cplxi R^{E\otimes \Lambda^{n-p,0}TM} \evwedge \Lambda$)
 can be estimated from above by
 \begin{multline}
  \label{eq:spectral_percolation_localized_proof_1}
  2\big\|d^{E\otimes \Lambda^{n-p,0}TM,\fadj}_{1,0} \widetilde u\big\|^2
  + \frac{q-1}{q}\big\llangle \Cplxi R^{E\otimes \Lambda^{n-p,0}TM} \evwedge \Lambda \widetilde u,\widetilde u \big\rrangle + \\
  + \frac{q-1}{q}\int_{\partial M \cap U} \mathscr L(u,u)\,d\mu_{\partial M}
  + \frac{2n}{q}\Big(\max\big\{\|\nabla \ol w_j\|_{L^\infty}^2 : 1 \leq j \leq n\big\}\Big) \|u\|^2.
 \end{multline}
 By our pseudoconvexity assumption,
 $
 \frac{q-1}{q} \int_{\partial M \cap U} \mathscr L(u,u)\,d\mu_{\partial M} \leq 2 \int_{\partial M \cap U} \mathscr L(u,u)\,d\mu_{\partial M}
 $,
 and the curvature bound yields
 \begin{equation}
  \label{eq:spectral_percolation_localized_proof_3}
  \frac{q-1}{q}\big\llangle \Cplxi R^{E\otimes \Lambda^{n-p,0}TM} \evwedge \Lambda \widetilde u,\widetilde u \big\rrangle
  \leq 2 \big\llangle \Cplxi R^{E\otimes \Lambda^{n-p,0}TM} \evwedge \Lambda \widetilde u,\widetilde u \big\rrangle - c\bigg(2-\frac{q-1}{q}\bigg)\|u\|^2.
 \end{equation}
 Using this and again applying \cref{eq:bkn_boundary_localized}, we see that \cref{eq:spectral_percolation_localized_proof_1} is dominated by
 \[
  2Q^E_U(u,u) + \bigg(\frac{2n\kappa}{q} - c\bigg(2-\frac{q-1}{q}\bigg)\bigg)\|u\|^2 = 2Q^E_U(u,u) + \frac{2n\kappa -c (q+1)}{q}\,\|u\|^2,
 \]
 as claimed.
\end{proof}

\begin{proof}[Proof of \cref{spectral_percolation_specific}]
 From \cref{bounded_geometry_partitions_frames_kaehler}, we know that there are geodesic balls $\{B(x_k,r) : k \in \N\}$ that cover $X$ along with a subordinate partition of unity $(\varphi_k^2)_{k \in \N}$ satisfying the estimate $\gamma \coloneqq \sup_{k \in \N}\|d\varphi_k\|_{L^\infty}^2 < \infty$ and with orthonormal frames $(w_j^k)_{j=1}^n$ of $T^{1,0}X|_{B(x_k,r)}$ such that we have the uniform bound $|\nabla \ol w_j^k|^2 \leq \kappa < \infty$ for all $j$ and $k$.
 Fix $\lambda < \inf\essspec{\square^E_{p,q-1}}$ and take a compact subset $K \subseteq \interior M$ with $\inf\spec{\square^E_{M\setminus K,p,q-1}} \geq \lambda$, see \cref{limit_of_infimum_of_spectrum}.
 Let $u \in B_{\partial M}^{p,q}(M\setminus K,E)$.
 Applying the localization formula~\cref{eq:ims_localization_first_order} to $D \coloneqq \sqrt 2(\dbar^E + \dbar^{E,\fadj})|_{\Omega^{p,\bullet}(M,E)}$, with $\symb{D}(\gamma) = c_p(\gamma)$ the Clifford action from \cref{eq:clifford_kaehler}, we have
 \begin{align*}
   2 Q^E(u,u)
   & = \sum_{k = 1}^\infty \big(2 Q^E(\varphi_ku,\varphi_ku) - \|c_p(d\varphi_k) u\|^2 \big) \\
   & \geq \sum_{k = 1}^\infty \bigg(\frac{1}{q}\sum_{j=1}^n Q^E(\ins{\ol w_j^k}(\varphi_k u),\ins{\ol w_j^k}(\varphi_k u)) - C\|\varphi_k u\|^2 - \|c_p(d\varphi_k) u\|^2\bigg) \\
   & \geq \sum_{k = 1}^\infty\bigg(\frac{\lambda}{q} \sum_{j=1}^n \|\ins{\ol w_j^k}(\varphi_k u)\|^2 - C\|\varphi_k u\|^2\bigg) - \gamma N \|u\|^2 \\
   & = (\lambda - C - \gamma N) \|u\|^2,
 \end{align*}
 where $C$ is the constant from \cref{spectral_percolation_localized} (with $\kappa$ modified to take the supremum over $k$ also), and $N$ is the intersection multiplicity of the cover $\{B(x_k,r)\}_{k \in \N}$, see \cref{bounded_geometry_partitions_frames}.
 By \cref{form_core_for_square,limit_of_infimum_of_spectrum}, this implies $2\inf\essspec{\square^E_{p,q}} \geq \lambda - C - \gamma N$, from which the claim follows.
\end{proof}

\begin{rem}
 \begin{inlineenum}
 \item\label{rem:spectral_percolation_curvature_assumption}
  If $M \subseteq X$ in \cref{spectral_percolation_specific} is bounded (hence compact, since it is complete by assumption), then the curvature condition on $E|_M$ in \cref{spectral_percolation_specific} is of course vacuous, and $X$ also does not have to be of bounded geometry anymore.

 \item
  In the constants that appear in \cref{spectral_percolation_specific} and are explicit in its proof, it is apparent that one may replace $\nakanoc[q]{E\otimes \Lambda^{n-p,0}TM}$ with $\lim_K \nakanoc[q]{E|_{M\setminus K} \otimes \Lambda^{n-p,0}(T(M\setminus K))}$, with $K$ ranging over the compact subsets of $\interior M$.
  
 \item
  If $M$ is a (Levi) pseudoconvex domain in $\Cplx^n$ with smooth boundary, and $E \to M$ is a Nakano semipositive vector bundle, then retracing the proof of \cref{spectral_percolation_localized}, we find
  \[ \frac{1}{q} \sum_{j=1}^n Q^E(\ins{\ol w_j}(u),\ins{\ol w_j}(u)) \leq Q^E(u,u) \]
  for all $u \in B_{\partial M}^{p,q}(M,E)$, with $(w_j)_{j=1}^n$ some constant global orthonormal frame of $T^{1,0}M \cong M \times \Cplx^n$, since all the terms involving estimates of the derivatives of $w_j$ do not appear.
  Consequently, the condition $\inf\essspec{\square^E} > 0$ also percolates up the $\dbar^E$-complex in this case.
  This is included in the orginal result of Fu from \cite[Proposition~2.2]{Fu2008}.
 \end{inlineenum}
\end{rem}

 \section{Proofs of \cref{necessary_condition_for_schrodinger_operator_discrete_spectrum,necessary_condition_square_discrete_spectrum}}
\label{sec:schrodinger_operators}

Suppose that $\nabla$ is a metric connection on $E$ and $V \colon E \to E$ a self-adjoint bundle endomorphism and let $H_{\nabla,V}$ be the generalized Schr\"odinger operator from~\cref{eq:schrodinger_type_operator}.
We will always make the assumption that $H_{\nabla,V}$ is lower semibounded.
In case $M$ is complete and without boundary, this implies that $H_{\nabla, V}$ is essentially self-adjoint, see \cite[Theorem~2.13]{Braverman2002}.
For $U \subseteq M$ an open subset, define
\newcommand{\energy}[2]{{\mathcal E_{#1}(#2)}}
\begin{equation}
 \label{eq:schrodinger_energy}
 \energy{\nabla,V}{U} \coloneqq \inf\bigg\{\frac{\llangle H_{\nabla,V}s,s \rrangle}{\|s\|^2} : s\in \bseccc(M,E)\setminus\{0\}\text{ with } \supp(s) \subseteq U\bigg\}.
\end{equation}

\begin{rem}
 \label{rem:energy_and_essential_spectrum}
 \begin{inlineenum}
  \item
   Suppose that $A \colon \dom{A} \subseteq L^2(M,E) \to L^2(M,E)$ is a lower semibounded self-adjoint extension of $H_{\nabla,V}$.
   Then $\energy{\nabla,V}{U} \geq \inf\spec{A_U}$ for every open subset $U$ of $M$, where $A_U$ is defined in \cref{def:localized_differential_operator}.
   From \cref{limit_of_infimum_of_spectrum}, we obtain
   \begin{equation}
    \label{eq:lim_energy_essspec}
    \lim_K \energy{\nabla,V}{M\setminus K} \geq \inf\essspec{A},
   \end{equation}
   with $K$ ranging over the compact subsets of $\interior M$.

  \item
   Let $(U_n)_{n \in \N}$ be a sequence of open subsets of $M$ with $U_n \to \infty$ as $n \to \infty$, meaning that for all compact $K \subseteq \interior M$ there is $n_0 \in \N$ such that $U_n \subseteq M\setminus K$ for all $n \geq n_0$.
   Then
   \begin{equation}
    \label{eq:liminf_energy_open_subsets_to_infty}
    \liminf_{n \to \infty} \energy{\nabla,V}{U_n} \geq \inf\essspec{A}.
   \end{equation}
   Indeed, let $\lambda$ be an accumulation point of $n \mapsto \energy{\nabla,V}{U_n}$, with $\lim_{k \to \infty} \energy{\nabla,V}{U_{n_k}} = \lambda$ for some subsequence $k \mapsto U_{n_k}$.
   Without loss of generality, we can assume that $U_{n_k} \subseteq M\setminus K_k$, where $(K_k)_{k \in \N}$ is an exhaustion of $\interior M$ by compact subsets.
   It follows from \cref{eq:lim_energy_essspec} that
   \[ \lambda = \lim_{k \to \infty} \energy{\nabla,V}{U_{n_k}} \geq \lim_{k \to \infty} \energy{\nabla,V}{M\setminus K_k} = \lim_K \energy{\nabla,V}{M\setminus K} \geq \inf\essspec{A}. \qedhere \]
 \end{inlineenum}
\end{rem}

The following result and its proof are motivated by \cite[Main Theorem]{Iwatsuka1986} (see also \cite[Theorem~6.10]{Shubin1999}):

\begin{lem}
 \label{schrodinger_operator_bounded_geometry_ball_energy}
 Let $M$ be a Riemannian manifold of $1$-bounded geometry (without boundary), $E \to M$ a Hermitian vector bundle, $\nabla$ a connection on $E$, and $V$ a self-adjoint bundle endomorphism of $E$.
 Assume that $H_{\nabla,V}$ is lower semibounded, hence essentially self-adjoint on $\bsecc(M,E)$.
 Then the following are equivalent:
 \begin{enumerate}
 \item\label{item:schrodinger_operator_bounded_geometry_ball_energy_discrete_spectrum}
  The closure of $H_{\nabla,V}$ has discrete spectrum.
 \item\label{item:schrodinger_operator_bounded_geometry_ball_energy_ball_energy}
  $\lim_{k \to \infty} \energy{\nabla,V}{B(x_k,r)} = \infty$ for all sequences $x_k \in M$ with $x_k \to \infty$ as $k \to \infty$ and all $r>0$ small enough.
 \end{enumerate}
\end{lem}

\begin{proof}
 If the spectrum of $\ol{H_{\nabla,V}}$ is discrete, then condition~\cref{item:schrodinger_operator_bounded_geometry_ball_energy_ball_energy} holds because of \cref{eq:liminf_energy_open_subsets_to_infty}.
 Conversely, suppose that \cref{item:schrodinger_operator_bounded_geometry_ball_energy_ball_energy} is true.
 We show that there is a proper smooth function $\psi \colon M \to [C,\infty)$, where $C \in \R$ will be determined later, such that $\llangle H_{\nabla, V} s,s \rrangle \geq \int_M \psi |s|^2\,d\mu_g$ for all $s \in \bsecc(M,E)$, from which the claim follows by using \cref{characterization_of_discreteness_of_spectrum} and essential self-adjointness of $H_{\nabla,V}$.

 If $M$ is compact, there is nothing to show due to \cref{eq:lim_energy_essspec}, so we may assume that $M$ is noncompact.
 Let $\{B(x_k,r)\}_{k \geq 1}$ be a countable cover of $M$ by geodesic balls as in \cref{bounded_geometry_partitions_frames}, with associated functions $\varphi_k \in C^\infty(M,[0,1])$.
 Then $x_k \to \infty$ as $k \to \infty$, for if a subsequence would stay in a compact subset of $M$, it would have a limit point in $M$, contradicting the fact that this cover has uniformly finite intersection multiplicity, see \cref{bounded_geometry_partitions_frames}.
 By definition, $\llangle H_{\nabla,V}(\varphi_k s),\varphi_k s\rrangle \geq \energy{\nabla,V}{B(x_k,r)}\|\varphi_k s\|^2$, hence
 \[
  \llangle H_{\nabla,V}s,s \rrangle \geq \int_M \sum_{k=1}^\infty \Big(\energy{\nabla,V}{B(x_k,r)}\varphi_k^2 - |d\varphi_k|^2\Big)|s|^2\,d\mu_g
 \]
 follows from the IMS localization formula~\cref{eq:ims_localization}.

 Let $\psi \colon M \to \R$ denote the function defined by the series.
 Then $\psi$ is smooth and maps $M$ to $[C,\infty)$, where $C \coloneqq \inf\spec{H_{\nabla, V}} - N \gamma$, with $N$ the intersection multiplicity of the cover $\{B(x_k,r)\}_{k \geq 1}$, and $\gamma \coloneqq \sup_{k \in \N}\|d\varphi_k\|_{L^\infty}$.
 Moreover, $\psi \colon M \to [C,\infty)$ is proper:
 if $\lambda \in \R$, then we find $k_0 \in \N$ such that $\energy{\nabla,V}{B(x_k,r)} \geq \lambda$ for all $k \geq k_0$, \ie $\psi \geq \lambda - N\gamma$ on $\bigcup_{k \geq k_0} B(x_k,r)$, a set whose complement is bounded, hence with compact closure by the Hopf--Rinow theorem.
 This completes the proof.
\end{proof}

In what follows, we are mostly concerned with Schr\"odinger operators acting on sections of Hermitian \emph{line} bundles, and where the connection is a metric connection.
Because $\en{L}$ is trivial, we can identify $V$ with a smooth function on $M$, and $\Omega^1(M,\en{L})$ with $\Omega^1(M,\mathbb C)$.
We will use the fact that the set of metric connections on a given line bundle $L \to M$ may be described as the affine space $\{\nabla^0 + \Cplxi\alpha\otimes \id{L} : \alpha \in \Omega^1(M,\mathbb R)\}$ for any reference metric connection $\nabla^0$ on $L$.

\begin{lem}[Gauge invariance]
 \label{gauge_invariance_of_energy}
 Let $U \subseteq M$ be a simply connected open subset.
 Then $\energy{\nabla,V}{U} = \energy{\nabla',V}{U}$ for any two metric connections $\nabla$ and $\nabla'$ on $L|_U$ with the same curvature.
\end{lem}

\begin{proof}
 This is just a geometric reinterpretation of the corresponding property of scalar Schr\"od\-ing\-er operators on $\R^n$, see for instance \cite[Theorem~1.2]{Leinfelder1983}.
 The difference of the two metric connections is a purely imaginary one-form, \ie $\nabla - \nabla' = \Cplxi\alpha \otimes \id{L}$ with $\alpha \in \Omega^1(U,\mathbb R)$.
 Since the curvatures agree, we have $d\alpha = 0$.
 Indeed, $d^\nabla = d^{\nabla'} + \Cplxi \extp{\alpha}$, and hence
 \[ R^\nabla \evwedge s = d^\nabla (\nabla s) = d^{\nabla'}(\nabla' s + \Cplxi \alpha \otimes s) + \Cplxi \alpha \wedge(\nabla ' s + \Cplxi \alpha \otimes s) = R^{\nabla'}\evwedge s + \Cplxi d\alpha \otimes s \]
 for all $s \in \bsec(U,L)$.
 Because $U$ is simply connected, de~Rham's theorem implies that there is $g \in C^\infty(U,\mathbb R)$ such that $\alpha = dg$.
 For $s \in \bsecc(M,E)$ with support in $U$, we compute
 \[ \nabla(e^{-\Cplxi g}s) = -\Cplxi e^{-\Cplxi g} dg \otimes s + e^{-\Cplxi g}(\nabla' s + \Cplxi dg \otimes s) = e^{-\Cplxi g}\,\nabla' s, \]
 hence
 $\llangle H_{\nabla,V} (e^{-\Cplxi g}s),e^{-\Cplxi g}s \rrangle = \llangle H_{\nabla',V}s,s \rrangle$
 and therefore $\energy{\nabla,V}{U} = \energy{\nabla',V}{U}$.
\end{proof}

The following \namecref{poincare_lemma_with_estimates} extends \cite[Proposition~3.2]{Iwatsuka1986} to Riemannian manifolds of $0$-bounded geometry:

\begin{lem}
 \label{poincare_lemma_with_estimates}
 Let $M$ be a Riemannian manifold of $0$-bounded geometry.
 There exists $\rho > 0$ with the following property:
 if $r \in (0,\rho)$, $x \in M$, and $B \in \Omega^2(\ol{B(x,r)})$ is a closed two-form, then there is $a \in \Omega^1(B(x,r))$ such that $da = B$ and
 \[ \|a\|_{L^p(B(x,r),T^*M)} \leq C_p(r)\|B\|_{L^p(B(x,r),\Lambda^2 T^*M)} \]
 for all $1 \leq p \leq \infty$, where $C_p(r) > 0$ depends only on $p$, $r$, and on the geometry of $M$, but \emph{not} on $x$.
\end{lem}

\begin{proof}
 Let $\rho > 0$ be such that the distortion of normal coordinates on balls of radius at most $\rho$ is uniformly bounded on $M$, see \cref{bounded_geometry_bounded_distortion}.
 Take $B \in \Omega^2(B(x,r))$ as in the assumption and put $\widetilde B \coloneqq \varphi_x^*B$, where $\varphi_x \coloneqq (\exp_x \circ \tau)|_{B_{\R^n}(0,r)}$ are Riemannian normal coordinates, with $\tau \colon \R^n \to T_xM$ any orthonormal map (\ie a choice of orthonormal basis of $T_xM$), chosen in a way that $\varphi_x$ preserves the orientations.
 Then $\widetilde B$ is an element of $\Omega^2(\ol{B_{\R^n}(0,r)})$, closed by naturality  of the exterior derivative ($d$ commutes with pullbacks), and the construction in \cite[Proposition~3.2]{Iwatsuka1986} yields $\widetilde a \in \Omega^1(B_{\R^n}(0,r))$ such that $da = B$ and
 \[ \|\widetilde a\|_{L^p(B_{\R^n}(0,r),T^*\R^n)} \leq \widetilde C_p(r) \|\widetilde B\|_{L^p(B_{\R^n}(0,r),\Lambda^2 \R^n)} \]
 for all $1 \leq p \leq \infty$.
 This can be achieved by taking
 \[ \widetilde a_{y}(v) \coloneqq \fint_{B_{\R^n}(0,r)} \int_0^1 \widetilde B_{z+t(y-z)}(t(y-z),v)\,dt\,d\lambda(z), \]
 with $y \in B_{\R^n}(0,r)$ and $v \in T_y(B_{\R^n}(0,r)) \cong \R^n$, and where $\fint \arghere\,d\lambda$ denotes the average with respect to Lebesgue measure.
 Define $a \coloneqq (\varphi_x^{-1})^* \widetilde a$.
 Then $da = B$, again by naturality, and we have
 \[
 |\varphi_x^*a|(y) \geq |T_y \varphi_x|^{-1}\, (|a| \circ \varphi_x)(y)
 \quad\text{and}\quad
 |\varphi_x^*B|(y) \leq |T_y \varphi_x|^2\, (|B| \circ \varphi_x)(y),
 \]
 where $|T_y\varphi_x|$ is the operator norm.
 By \cref{bounded_geometry_bounded_distortion,rem:bounded_geometry_bounded_coefficients}, there is $C > 0$ such that $1/C \leq |T_y \varphi_x| \leq C$ and $1/C \leq \det(g^x_{ij}(y))^{1/2} \leq C$ uniformly in $y \in B_{\R^n}(0,r)$, and independent of $x \in M$.
 Here, $g^x_{ij}$ are the metric coefficients with respect to the chart $\varphi_x$.
 Putting this together, we obtain
 \begingroup
 \allowdisplaybreaks
 \begin{align*}
   \int_{B(x,r)} |a|^p\,\volg
   & = \int_{B_{\R^n}(0,r)} (|a|^p \circ \varphi_x) \,\varphi_x^*\volg \\
   & \leq C^{p+1} \int_{B_{\R^n}(0,r)} |\varphi_x^*a|^p(y)\,d\lambda(y) \\
   & \leq C^{p+1} \widetilde C_p(r)^p \int_{B_{\R^n}(0,r)}|\varphi_x^* B|^p(y)\,d\lambda(y) \\
   & = C^{3p+2} \widetilde C_p(r)^p \int_{B(x,r)}|B|^p\,\volg,
 \end{align*}
 \endgroup
 with $\lambda$ the Lebesgue measure on $\R^n$.
 Now put $C_p(r) \coloneqq C^{3+2/p} \widetilde C_p(r)$.
\end{proof}

Consider now a local trivialization $\psi \colon p^{-1}(U) \xrightarrow{\cong} U \times \mathbb C$ of $L$ over an open subset $U \subseteq M$.
Then there is $\alpha_\psi \in \Omega^1(U,\mathbb C)$ such that
\[ ((\id{T^*U} \otimes \psi) \circ \nabla \circ \psi^{-1}) f = (d + \alpha_\psi)f \]
for every function $f \in C^\infty(U) = \bsec(U,U\times \mathbb C)$.
For the exterior covariant derivative, this means $(\id{\Lambda T^*U} \otimes \psi) \circ d^\nabla \circ (\id{\Lambda T^*U} \otimes \psi^{-1}) = d + \extp{\alpha_\psi}$ on $\Omega(U)$.
Note that the curvature of $\nabla$ on $U$ is given by
\begin{equation}
 \label{eq:line_bundle_curvature_trivialization}
 R^\nabla|_U = d\alpha_\psi \otimes \id{L} \in \Omega^2(U,\en{L}).
\end{equation}
If $L$ carries a Hermitian metric, then there is a smooth function $w_\psi \colon U \to \R$ such that $|\psi^{-1}(y,\lambda)| = |\lambda|e^{-w_\psi(y)}$ for all $(y,\lambda) \in U \times \Cplx$.

\begin{lem}
 \label{integral_of_connection_form_potential_to_infty}
 Let $\nabla$ be a metric connection on a Hermitian line bundle $L \to M$, and let $V \colon L \to L$ be a self-adjoint vector bundle morphism.
 Suppose that $U \subseteq M$ is open and contractible,  and $\varphi \colon M \to [0,1]$ is smooth with $\supp(\varphi) \subseteq U$.
 Then
 \[ \inf_{g \in C^\infty(U,\mathbb R)} \int_U \big(|\alpha_\psi - dw_\psi + idg|^2 + |V|\big)\,d\mu_g \geq \energy{\nabla,V}{U} \|\varphi\|^2_{L^2(M)} - \|d\varphi\|^2_{L^2(M,T^*M)}, \]
 where $\psi$ is any local trivialization of $L$ over $U$, and where $\alpha_\psi \in \Omega^1(U,\Cplx)$ and $w_\psi \in C^\infty(U,\R)$ are as above.
\end{lem}

\begin{proof}
 The proof is a modification of \cite[Lemma~5.1]{Iwatsuka1986} to accommodate globally nontrivial line bundles.
 Because $U$ is contractible, $L|_U$ is trivial, see for instance \cite[p.~15]{Moore2001}.
 Let $\psi \colon p^{-1}(U) \to U \times \mathbb C$ be a local trivialization of $L$, and let $W \colon U \times \mathbb C \to U \times \mathbb C$ be the vector bundle isomorphism $(y,\lambda) \mapsto (y,e^{-w_\psi(y)}\lambda)$.
 Then $\psi_0 \coloneqq W\circ \psi$ is also a local trivialization of $L$ over $U$, and $|\psi_0^{-1}(y,\lambda)|_L = |\lambda|$.
 It follows that $(\id{T^*M}\otimes \psi_0)^{-1}\circ d \circ \psi_0$ is a metric connection on $L|_{U}$.
 Since
 \begin{multline*}
  \nabla|_{U} = (\id{T^*M}\otimes \psi_0)^{-1}\circ (\id{T^*M}\otimes W)\circ (d + \alpha_\psi) \circ W^{-1} \circ \psi_0 = \\
  = (\id{T^*M}\otimes \psi_0)^{-1} \circ (d + \alpha_\psi - dw_\psi) \circ \psi_0
 \end{multline*}
 and $\nabla$ is a metric connection, we see that $\Cplxi(\alpha_\psi - dw_\psi) \in \Omega^1(U,\mathbb R)$.
 Put
 \[ s \coloneqq \psi_0^{-1} \circ (\id{U} ,\varphi|_{U}) \colon U \to L, \]
 so that $s$ is a compactly supported section of $L$ over $U$ which extends to a section of $L$ over $M$ by setting it to zero outside of $\supp(\varphi)$.
 Evidently, $|s|_L^2 = |\varphi|^2 \leq 1$.
 Moreover, for $g \in C^\infty(U,\mathbb R)$, the connection $\nabla' \coloneqq \nabla|_{U} + \Cplxi dg \otimes \id{L}$ on $L|_{U}$ is metric compatible, and
 \begin{multline*}
  |\nabla' s|_{T^*M \otimes L}^2
  = |d\varphi + \varphi(\alpha_\psi - dw_\psi + \Cplxi dg)|^2 = \\
  = |d\varphi|^2 + |\varphi(\alpha_\psi - dw_\psi + \Cplxi dg)|^2
  \leq |d\varphi|^2 + |\alpha_\psi - dw_\psi + \Cplxi dg|^2,
 \end{multline*}
 since the expression in the parentheses is purely imaginary,
 and $|\varphi| \leq 1$.
 Because $ddg = 0$, we have $R^{\nabla'} = R^{\nabla|_{U}}$, and \cref{gauge_invariance_of_energy} implies
 \begin{multline*}
   \int_U \big(|\alpha_\psi - dw_\psi + idg|^2 + |V|\big)\,d\mu_g + \|d\varphi\|_{L^2(M,T^*M)}^2 \geq \\
   \geq \int_{U} \big(|\nabla' s|_{T^*M\otimes L}^2 + \langle Vs,s\rangle_L\big)\,d\mu_g
   = \llangle H_{\nabla',V}s,s \rrangle \geq \\
   \geq \energy{\nabla',V}{U}\,\|s\|_{L^2(M,L)}^2
   = \energy{\nabla,V}{U}\,\|\varphi\|_{L^2(M)}^2.
 \end{multline*}
 Since $\psi$ and $g \in C^\infty(U,\R)$ were arbitrary, the claim follows.
\end{proof}

We are now ready to show \cref{necessary_condition_for_schrodinger_operator_discrete_spectrum,necessary_condition_square_discrete_spectrum}.

\begin{proof}[Proof of \cref{necessary_condition_for_schrodinger_operator_discrete_spectrum}]
 It suffices to prove the claim for $r>0$ small enough, and we take $r$ so that item~\cref{item:schrodinger_operator_bounded_geometry_ball_energy_ball_energy} of \cref{schrodinger_operator_bounded_geometry_ball_energy} and \cref{poincare_lemma_with_estimates} work out.
 Let $(x_k)_{k \in \N}$ be a sequence in $M$ with $x_k \to \infty$ as $k \to \infty$.
 For every $k \in \N$, we find $\varphi_k \in C^\infty(M,[0,1])$ with $\supp(\varphi_k) \subseteq B(x_k,r)$, $\int_M |\varphi_k|^2\,d\mu_g = 1$, and such that $\sup_{k \in \N}\|d\varphi_k\|_{L^\infty(M,T^*M)} < \infty$, see \cref{bounded_geometry_bump_functions}.
 Since $\nabla$ is a metric connection, we have  $R^\nabla = d\alpha_{\psi_k} \otimes \id{L}$ on $B(x_k,r)$ with $\Cplxi \alpha_{\psi_k} \in \Omega^1(B(x_k,r),\mathbb R)$ for any choice of local trivializations $\psi_k \colon L|_{B(x_k,r)} \to B(x_k,r) \times \mathbb C$, see \cref{eq:line_bundle_curvature_trivialization}.
 By \cref{poincare_lemma_with_estimates}, there are $a_k \in \Omega^1(B(x_k,r),\mathbb R)$ with $da_k = id\alpha_{\psi_k}$ and
 \[ \int_{B(x_k,r)} |R^\nabla|^2\,d\mu_g = \int_{B(x_k,r)} |d\alpha_{\psi_k}|^2\,d\mu_g \geq C \int_{B(x_k,r)} |a_k|^2\,d\mu_g, \]
 with $C>0$ independent of $x \in M$.
 Since $da_k = \Cplxi d(\alpha_{\psi_k} - dw_{\psi_k})$ and $B(x_k,r)$ is simply connected, there is $g_k \in C^\infty(B(x_k,r),\mathbb R)$ such that $a_k - \Cplxi \alpha_{\psi_k} + \Cplxi dw_{\psi_k} = dg_k$, \ie $a_k = \Cplxi \alpha_{\psi_k} - \Cplxi dw_{\psi_k} - dg_k$.
 Using \cref{integral_of_connection_form_potential_to_infty}, we find
 \[ \int_{B(x_k,r)} \big(C^{-1} |R^\nabla|^2 + |V|\big)\,d\mu_g \geq \energy{\nabla,V}{B(x_k,r)} \|\varphi_k\|^2_{L^2(M)} - \|d\varphi_k\|^2_{L^2(M,T^*M)}. \]
 If $A$ denotes a lower semibounded self-adjoint extension of $H_{\nabla,V}$ with discrete spectrum, then we have $\liminf_{k \to \infty}\energy{\nabla,V}{B(x_k,r)} \geq \inf\essspec{A} = \infty$ by \cref{eq:liminf_energy_open_subsets_to_infty}, so the claim follows.
\end{proof}

\begin{proof}[Proof of \cref{necessary_condition_square_discrete_spectrum}]
 Let $q \in \{0,n\}$.
 By~\cref{eq:dolbeault_weitzenboeck}, we have $\square^L_{p,q} = \Delta^{\Lambda^{p,q}T^*M\otimes L} + c_p(R^{\Lambda^{p,\bullet}T^*M\otimes L})$, where $c_p$ is the Clifford action on $\Lambda^{p,\bullet}T^*M \otimes L$ from~\cref{eq:clifford_kaehler}.
 By~\cref{eq:general_weitzenboeck_curvature_term},
 \begin{multline}
  \label{eq:bound_of_clifford_remainder}
  |c_p(R^{\Lambda^{p,\bullet}T^*M \otimes L})|
  \leq \sum_{j<k} |R^{\Lambda^{p,\bullet}T^*M \otimes L}(e_j,e_k)| \leq \\ 
  \leq \sqrt{n(2n+1)}\, \bigg(\sum_{j < k} |R^{\Lambda^{p,\bullet}T^*M \otimes L}(e_j,e_k)|^2\bigg)^{1/2} 
  \leq \sqrt{n(2n+1)} \, |R^{\Lambda^{p,\bullet}T^*M \otimes L}|,
 \end{multline}
 so if the spectrum of $\square^L_{p,n}$ is discrete, then \cref{necessary_condition_for_schrodinger_operator_discrete_spectrum} gives
 \[ \lim_{x \to \infty} \int_{B(x,r)}\Big(|R^{\Lambda^{p,\bullet}T^*M \otimes L}|^2 + \sqrt{n(2n+1)} \, |R^{\Lambda^{p,\bullet}T^*M \otimes L}|\Big)\,d\mu_g = \infty \]
 for all $r>0$ small enough.
 Now by H\"older's inequality,
 \begin{equation*}
  \int_{B(x,r)} |R^{\Lambda^{p,\bullet}T^*M \otimes L}|\,d\mu_g
  \leq \sqrt{C} \bigg(\int_{B(x,r)} |R^{\Lambda^{p,\bullet}T^*M \otimes L}|^2\,d\mu_g\bigg)^{1/2},
 \end{equation*}
 with $C \coloneqq \sup_{x \in M} \mu_g(B(x,r))$. \footnote{   The supremum is finite because in normal coordinates around $x$ and with small enough radius, the metric coefficients $g_{ij}^x$ have uniform two-sided bounds, independent of $x$, see \cref{rem:bounded_geometry_bounded_coefficients}, hence $\mu_g(B(x,r)) = \int_{B_{\R^n}(0,r)} \det(g_{ij}^x)^{1/2}\,d\lambda$ is also bounded from both sides.
 }
 Consequently,
 \[ \int_{B(x,r)} |R^{\Lambda^{p,\bullet}T^*M \otimes L}|^2\,d\mu_g \to \infty \quad\text{as } x \to \infty, \]
 which is the same as \cref{eq:necessary_condition_square_discrete_spectrum} since the curvature of $\Lambda^{p,\bullet}T^*M$ is bounded due to $M$ having $0$-bounded geometry.
 In the case where $1 \leq q \leq n-1$ and $L$ is $(q+1)$-Nakano lower semibounded, we use \cref{spectral_percolation_specific} to reduce this case to the first one, see also \cref{rem:nakano_lower_semibounded}.
 Finally, if $\square^L_{p,q}$ has discrete spectrum and $L^*$ is $(n-q+1)$-Nakano lower semibounded, then $\square^L_{p,0}$ has discrete spectrum by \cref{spectral_percolation_serre}, so \cref{eq:necessary_condition_square_discrete_spectrum} also follows.
\end{proof}

\appendix
\ifams\else
 \gdef\thesection{\Alph{section}}  \makeatletter
 \makeatother
\fi

\section{Riemannian manifolds of bounded geometry}
\label{sec:bounded_geometry}

Let $(M,g)$ be a Riemannian manifold.
In this \namecref{sec:bounded_geometry}, we will only consider the case where $M$ has no boundary.
For $p \in M$, we denote by $\exp_p \colon \mathscr D_p \subseteq T_pM \to M$ the (Riemannian) exponential map.
The \defemph{injectivity radius of $(M,g)$ at a point $p \in M$} is the supremum of all $r > 0$ such that $\exp_p$ restricts to a diffeomorphism on $B_{T_pM}(0,r)$, where $B_{T_pM}(0,r)$ is the open ball in $(T_pM,g_p)$ around $0$ and with radius $r$.
The image of this ball under $\exp_p$ is then $B(p,r) \coloneqq \{q \in M : d_g(p,q) < r\}$, the open ball for the Riemannian distance $d_g$.
The \defemph{injectivity radius of $(M,g)$}, denoted by $\rinj{M,g}$, is the infimum over all injectivity radii at points $p \in M$.

\begin{defn}
 \label{def:bounded_geometry}
 A connected Riemannian manifold $(M,g)$ is said to be of \defemph{$k$-bounded geometry} if its injectivity radius $\rinj{M,g}$ is positive, and there exist constants $C_j > 0$ such that $|\nabla^j R^M| \leq C_j$ for all $0 \leq j \leq k$, where $\nabla^j R^M$ is the $j$\textsuperscript{th} covariant derivative of the Riemannian curvature tensor of $M$.
 If $(M,g)$ is of $k$-bounded geometry for all $k \in \N$, then it is said to be of \defemph{bounded geometry}.
\end{defn}

\begin{rem}
 \label{rem:bounded_geometry}
 \begin{inlineenum}
 \item
  All Riemannian manifolds of $k$-bounded geometry are complete due to the bound on the injectivity radius, see \cite[Proposition~2.2]{Eichhorn2008}.

 \item
  There is also a notion of bounded geometry for vector bundles:
  a Hermitian (or Riemannian) vector bundle $E \to M$ with metric connection $\nabla$ is called a \defemph{Hermitian (Riemannian) vector bundle of $k$-bounded geometry} if $M$ is a Riemannian manifold of $k$-bounded geometry, and the curvature of $\nabla$ satisfies $|\nabla^j R^\nabla| \leq C_j$ for all $0 \leq j \leq k$, uniformly on $M$.
  Again, $E$ is said to be of \emph{bounded geometry} if this holds for all $k \in \N$.
   Most prominently, the tangent bundle as well as all tensor bundles of a manifold of bounded geometry (with the Levi--Civita connection) are vector bundles of bounded geometry \cite[p.~45]{Eldering2013}.
 \end{inlineenum}
\end{rem}

Manifolds of bounded geometry come with a nice cover by open subsets, namely the geodesic balls $B(p,r)$ for fixed $r < \rinj{M,g}$ small enough, see
\cref{bounded_geometry_partitions_frames} below.
Recall that any choice of orthonormal basis $\{e_j\}_{j=1}^n$ of $T_pM$, with $p \in M$ fixed, gives rise to a chart of $M$ via
$$ B(p,r) \to B_{\R^n}(0,r) \subseteq \R^n, \quad q \mapsto (\exp_p \circ \tau)^{-1}(q), $$
where $\tau \colon \R^n \to T_pM$ is the isometry $\tau(t_1,\dotsc,t_n) \coloneqq t_1 e_1 + \dotsb + t_n e_n$.
These charts are called \defemph{(Riemannian) normal coordinates}.
The following \cref{bounded_geometry_bounded_distortion} will show that the \defemph{distortion} of normal coordinates can be uniformly bounded on a manifold of $0$-bounded geometry.
An explicit statement of this fact can be found in \cite[Lemma~2.2]{Roe1988}.

\begin{lem}
 \label{bounded_geometry_bounded_distortion}
 Let $(M,g)$ be a Riemannian manifold with positive injectivity radius and sectional curvatures uniformly bounded, \ie
 $|K(\Pi)| \leq C$
 for some $C>0$ and for all two-dimensional subspaces $\Pi \subseteq T_p M$ and every $p \in M$.
 Then there exist $0 < r < \rinj{M,g}$ and $C_1,C_2 > 0$ such that
 \begin{equation}
  \label{eq:bounded_geometry_bounded_distortion_proof_1}
  C_1 |X| \leq |(T_v \exp_p)X| \leq C_2|X|
 \end{equation}
 for all $p \in M$, all $0 \neq v \in B_{T_pM}(0,r)$, and all $X \in T_pM$.
\end{lem}

\begin{proof}
 The bounds~\cref{eq:bounded_geometry_bounded_distortion_proof_1} can be obtained as a consequence of the Rauch comparison theorem (see \cite[p.~215]{Carmo1992} for a proof), applied to the Jacobi field $J(t) \coloneqq t(T_{tv/|v|}\exp_p)(X)$ along the geodesic $\gamma(t) \coloneqq \exp_{p}(tv/|v|)$ and comparing $M$ with the space forms of constant sectional curvature $\pm C$.
\end{proof}

\begin{rem}
 \label{rem:bounded_geometry_bounded_coefficients}
 \begin{inlineenum}
 \item
  Using \cref{bounded_geometry_bounded_distortion}, it is easy to see that if $(M,g)$ has uniformly bounded sectional curvature, then the coefficients $g_{ij}^p$ of the metric in normal coordinates $\varphi_p \coloneqq (\exp_p \circ \tau)^{-1}|_{B(0,r)}$ around a sufficiently small ball $B(p,r)$ are bounded from above and below, independent of $p$.
  Indeed, for $y \in B_{\R^n}(0,r)$, the coefficients $g_{ij}^p(y)$ are just the components of the bilinear form $(\tau^*\exp_p^*g)(y)$ on $\R^n$ with respect to the standard basis of $\R^n$.

 \item
  It is harder to argue that this also holds for derivatives of the metric coefficients:
  in \cite{Kaul1976}, it was shown that if $|R^M| \leq C_0$ and $|\nabla R^M| \leq C_1$, then also the Christoffel symbols with respect to normal coordinates (of sufficiently small radius) are bounded, uniformly in $p \in M$.
  Equivalently, the derivatives of the metric coefficients in such coordinates are also uniformly bounded.
  This was extended to arbitrary derivatives by Eichhorn in \cite[Corollary~2.6]{Eichhorn1991a}:
  if $(M,g)$ is open and complete and satisfies $|\nabla^j R^M| \leq C_j$ for $0 \leq j \leq k$, then the derivatives of order up to $k$ of the metric coefficients in normal coordinates around $p \in M$, and with sufficiently small radius $r$, are also bounded, uniformly in $p$.

 \item
  There is also a corresponding result for vector bundles, see \cite[Theorem~3.2]{Eichhorn1991a}.
  Assume that $(M,g)$ is of $k$-bounded geometry, and that $E \to M$ is a Hermitian vector bundle, equipped with a metric connection of $k$-bounded geometry, in the sense of \cref{rem:bounded_geometry}.
  Then there is $r>0$ and constants $\widetilde C_\gamma > 0$ such that
  \begin{equation}
   \label{eq:bounded_connection_coefficients_vector_bundle}
   |\partial^\gamma \Gamma^{\alpha}_{i\beta}| \leq \widetilde C_\gamma
  \end{equation}
  for all multiindices $|\gamma| \leq k-1$, all $1 \leq \alpha,\beta \leq \rank E$, and all $1 \leq i \leq \dim M$.
  Here, $\Gamma^\alpha_{i\beta}$ are the connection coefficients of the connection on $E$ with respect to a \defemph{synchronous framing}, \ie with respect to an orthonormal frame $(\xi^p_1,\dotsc,\xi^p_N)$ of $E|_{B(p,r)}$ obtained by parallel transporting an orthonormal basis of $E_p$ along the radial geodesics in $B(p,r)$.
  Thus, $\sum_{\beta=1}^n \Gamma^{\alpha}_{i\beta} \xi^p_\alpha = \nabla^E_{\partial_i} \xi^p_\beta$ or, equivalently, $\Gamma^{\alpha}_{i\beta} = \langle \nabla^E_{\partial_i} \xi^p_\beta , \xi^p_\alpha \rangle$, and the point is that the estimates~\cref{eq:bounded_connection_coefficients_vector_bundle} are again uniform in $p \in M$.
 \end{inlineenum}
\end{rem}

\begin{lem}
 \label{bounded_geometry_bump_functions}
 Let $(M,g)$ be a Riemannian manifold of $0$-bounded geometry.
 There exists $r \in (0,\rinj{M,g})$ and a constant $C>0$ with the following property:
 for all $p \in M$, there exists a smooth function $f_p \colon M \to [0,1]$ such that
 \begin{enumerate}
 \item
  $\supp(f_p) \subseteq B(p,r)$,
 \item
  $\|df_p\|_{L^\infty(M,T^*M)} \leq C$, and
 \item
  $C \geq \int_M |f_p|^2\,d\mu_g \geq 1/C$.
 \end{enumerate}
\end{lem}

\begin{proof}
 Take $r \in (0,\rinj{M,g})$ small enough such that the conclusion of \cref{bounded_geometry_bounded_distortion} holds, and such that the coefficients of the metric in normal coordinates on $B(p,r)$ are uniformly bounded, independent of $p$, see \cref{rem:bounded_geometry_bounded_coefficients}.
 Let $f \in C^\infty_c(B_{\R^n}(0,r),[0,1])$ be any nonzero function, and put $f_p \coloneqq f \circ \varphi_p^{-1} \colon B(p,r) \to [0,1]$, where $\varphi_p \coloneqq (\exp_p \circ \tau)|_{B_{\R^n}(0,r)}$ and $\tau \colon \R^n \to T_pM$ is an isometry such that $\varphi_p$ is orientation preserving.
 Then $f_p$ has compact support in $B(p,r)$, and we extend it by zero to all of $M$.
 For $X \in T_xM$, we have
 $$
 |df_p(X)|
 = |(T_x f_p)X|
 = |T_{\varphi_p^{-1}(x)} f \circ T_x \varphi_p^{-1} X|
 \leq \big|df(\varphi_p^{-1}(x))\big| \big|(T_x (\exp_p^{-1}))X\big|
 \leq C_1^{-1} \|df\|_{L^\infty} |X|
 $$
 by \cref{bounded_geometry_bounded_distortion}, hence $\|df_p\|_{L^\infty} \leq C_1^{-1} \|df\|_{L^\infty}$.
 Moreover,
 \begin{equation}
  \label{eq:bounded_geometry_bump_functions_proof_1}
   \widetilde C \|f\|_{L^2(B_{\R^n}(0,r))}^2 \geq
   \int_{B_{\R^n}(0,r)} |f(y)|^2\,\det(g_{ij}^p(y))^{1/2}\,d\lambda(y)
   \geq \widetilde C^{-1} \|f\|_{L^2(B_{\R^n}(0,r))}^2
 \end{equation}
 independent of $p$, with $\lambda$ the Lebesgue measure, and where $g_{ij}^p$ are the metric coefficients with respect to the normal coordinate chart $\varphi_p$, and the constant $\widetilde C$ is such that $1/\widetilde C \leq \det(g_{ij}^p)^{1/2} \leq \widetilde C$, \cf \cref{rem:bounded_geometry_bounded_coefficients}.
 Since the middle term in \cref{eq:bounded_geometry_bump_functions_proof_1} is $\int_M |f_p|^2\,\volg$, we are finished.
\end{proof}

\begin{prop}\label{bounded_geometry_partitions_frames}
 Let $(M,g)$ be a noncompact manifold of $1$-bounded geometry.
 Then there exists $r_0 \in (0,\rinj{M,g})$ such that for all $0 < r < r_0$ there is
 \begin{enumerate}
  \item\label{item:bounded_geometry_partitions_frames_cover}
   a countable cover $\{B(p_k,r)\}_{k \geq 1}$ of $M$ by geodesic balls, and a number $N > 0$ such that $\bigcap_{k \in J} B(p_k,r) \neq \emptyset$ implies $|J| \leq N$ for all subsets $J \subseteq \N$ (\ie the cover has uniformly finite intersection multiplicity),

  \item\label{item:bounded_geometry_partitions_frames_partition}
   a sequence of functions $\varphi_k \in C^\infty(M,[0,1])$ such that $\supp(\varphi_k) \subseteq B(p_k,r)$, $\sum_{k=1}^\infty \varphi_k^2 = 1$, and with $\sup_{k \in \N} \|d\varphi_k\|_{L^\infty} < \infty$, and

  \item\label{item:bounded_geometry_partitions_frames_frame}
   for every $k \in \N$, an orthonormal frame $(\xi_1^k,\dotsc,\xi_n^k)$ of $TM|_{B(p_k,r)}$ with
   $$ \sup_{k,j} \sup_{x \in B(p_k,r)} |\nabla \xi_j^k|_x < \infty. $$
 \end{enumerate}
\end{prop}

\begin{proof}
 For \cref{item:bounded_geometry_partitions_frames_cover,item:bounded_geometry_partitions_frames_partition}, see \cite[Lemma~2.16 and Corollary~2.18]{Eldering2013}, \cite[Lemma~1.2 and Lemma~1.3]{Shubin1992}, or \cite[Lemma~2.4]{Kaad2013}.
 Pick an orthonormal basis $(e_1^k,\dotsc,e_n^k)$ of $T_{p_k}M$, and denote by $(\xi_1^k,\dotsc,\xi_n^k)$ the frame of $TM|_{B(p_k,r)}$ that is obtained by parallel transporting the basis of $T_{p_k}M$ along the radial geodesics in $B(p_k,r)$.
 Then  $\nabla_X \xi_\alpha^k = \sum_{i,\beta} X^i \Gamma^\alpha_{i\beta}(x)\xi_\beta^k$, hence
 \begin{equation}
  \label{eq:bounded_geometry_partitions_frames_proof_1}
  |\nabla \xi_\alpha^k|_x
  = \sup_{|X|=1} |\nabla_X \xi_\alpha^k|_x
  \leq \sup_{|X|=1} \sum_{i,\beta} |X^i| |\Gamma^\alpha_{i\beta}(x) \xi_\beta^k|_x
  \leq \sup_{|X|=1} \sum_{i,\beta} |X^i| |\Gamma^\alpha_{i\beta}(x)|,
 \end{equation}
 where $\Gamma^\alpha_{i\beta}$ are the Christoffel symbols corresponding to the trivialization of $TM|_{B(p_k,r)}$ induced by the frame $(\xi_1^k,\dotsc,\xi_n^k)$ and the normal coordinates, and $X = X^i \partial_i$ with $\partial_i$ the normal coordinate vector fields.
 By the discussion about bundles of bounded geometry in \cref{rem:bounded_geometry_bounded_coefficients}, $|\Gamma^\alpha_{i\beta}(x)|$ is bounded by constants uniform in $x \in B(p_k,r)$, $k \in \N$, and $\alpha \in \{1,\dotsc,n\}$.
 Let $|\arghere|_e$ denote the Euclidean norm on $\R^n$.
 If $|X|=1$, then $|g(x)^{1/2} X|_e = 1$, where we view $g(x)$ as the symmetric matrix $(g_{ij}(x))_{i,j}$ (components in normal coordinates on $B(p_k,r)$), and $X$ as the vector $(X^1,\dotsc,X^n)$.
 It follows that
 \begin{equation}
  \label{eq:bounded_geometry_partitions_frames_proof_2} 
  |X^i| \leq |X|_e = \big|g(x)^{-1/2}g(x)^{1/2} X\big|_e \leq \big\|g(x)^{-1/2}\big\|_{\bdop{\R^n}} \big|g(x)^{1/2}X\big|_e = \big\|g(x)^{-1/2}\big\|_{\bdop{\R^n}}
 \end{equation}
 for $1 \leq i \leq n$, where $\|\arghere\|_{\bdop{\R^n}}$ is the operator norm.
 If $|g^{ij}| \leq C_0$ on $B(p_k,r)$ as in \cref{rem:bounded_geometry_bounded_coefficients}, then $\|g(x)^{-1}\|_{\bdop{\R^n}} \leq \trace(g(x)^{-1}) \leq nC_0$, and hence $\|g(x)^{-1/2}\|_{\bdop{\R^n}} \leq \sqrt{nC_0}$, uniformly in $x \in B(p_k,r)$, and not depending on $k$ and $r$.
 Combining this with \cref{eq:bounded_geometry_partitions_frames_proof_1,eq:bounded_geometry_partitions_frames_proof_2} finishes the proof.
\end{proof}

Since K\"ahler manifolds are also Riemannian manifolds, we may consider K\"ahler manifolds of bounded geometry.
The next result is just a simple adaptation of \cref{bounded_geometry_partitions_frames} to this case:

\begin{prop}
 \label{bounded_geometry_partitions_frames_kaehler}
 Let $X$ be a K\"ahler manifold of $1$-bounded geometry and complex dimension $n$, and let $\{B(p_k,r)\}_{k \geq 1}$ be a cover of $X$ as in \cref{bounded_geometry_partitions_frames}.
 Then for every $k \in \N$ there exists an orthonormal frame $(w_1^k,\dotsc,w_n^k)$ of $T^{1,0}X|_{B(p_k,r)}$ with
 $$ \sup_{k,j} \sup_{x \in B(p_k,r)} |\nabla w_j^k|_x < \infty. $$
 Moreover, $(\ol w_1^k,\dotsc,\ol w_n^k)$ is an orthonormal frame of $T^{0,1}X|_{B(p_k,r)}$ with the same boundedness property.
\end{prop}

\begin{proof}
 Choose an orthonormal basis $(e_1^k,\dotsc,e_n^k)$ of $T_{p_k}^{1,0}X$.
 Then $(\widetilde e_m^k)_{m=1}^{2n}$ defined by
 $$
 \widetilde e^k_{2j-1} \coloneqq \frac{1}{\sqrt{2}}(e^k_j + \ol e^k_j)
 \quad\text{and}\quad
 \widetilde e^k_{2j} \coloneqq J\widetilde e^k_{2j-1} = \frac{\Cplxi}{\sqrt{2}}(e^k_j - \ol e^k_j)
 \qquad(1 \leq j \leq n)
 $$
 is an orthonormal basis of $T_{p_k}X$, which we extend to an orthonormal frame $(\xi_1^k,\dotsc,\xi_{2n}^k)$ of $TX|_{B(p_k,r)}$ as in \cref{bounded_geometry_partitions_frames}.
 Since $X$ is K\"ahler, the complex structure $J$ is parallel for the Levi--Civita connection, see for instance \cite[Theorem~4.17]{Ballmann2006}.
 If $x \in B(p_k,r)$ and $\gamma$ denotes the radial geodesic from $p_k$ to $x$, then $\xi_m^k = P_\gamma(e_m^k)$ with $P_\gamma$ the parallel transport along $\gamma$, and
 therefore also $J(\xi_{2j-1}^k(x) - \Cplxi \xi_{2j}^k(x)) = \Cplxi(\xi_{2j-1}^k(x) - \Cplxi \xi_{2j}^k(x))$
 since the parallel transport commutes with the parallel endomorphism $J$.
 Hence,
 $$ w_j^k \coloneqq \frac{1}{\sqrt 2}(\xi_{2j-1}^k - \Cplxi \xi_{2j}^k) $$
 defines an orthonormal frame of $T^{1,0}X$ over $B(p_k,r)$, and with the required properties.
 The claim about $(\ol w_1^k,\dotsc, \ol w_n^k)$ is immediate.
\end{proof}

\ifams
 \printbibliography
\else
 \bibliography{resources/references}
\fi

\end{document}